\documentclass{svproc}
\usepackage[utf8]{inputenc}
\usepackage{amsmath}
\usepackage{amsfonts}
\usepackage{amssymb}
\usepackage[margin=4cm]{geometry}
\usepackage{graphicx}
\usepackage{footmisc}
\usepackage{multicol}
\usepackage{dsfont}
\usepackage{color}
\usepackage{hyperref}

\hypersetup{
    colorlinks=true,
    linkcolor=blue,
    filecolor=magenta,
    citecolor=blue,      
    urlcolor=blue,
    pdfpagemode=FullScreen,
    }
    
\usepackage{xcolor}                   
\definecolor{green}{HTML}{15522f}    
\definecolor{blue}{HTML}{2471a3 } 
\definecolor{red}{HTML}{E74C3C}
\definecolor{orange}{HTML}{d35400} 

\newcommand{\E}{\mathbb{E}}

\newcommand{\Prob}{\mathbb{P}}

\newcommand \Esp {\mathbb{E}}
\newcommand \Z {\mathbb{Z}}

\def\N{{\rm I\kern-0.16em N}}
\def\R{{\rm I\kern-0.16em R}}
\def\E{{\rm I\kern-0.16em E}}
\def\P{{\rm I\kern-0.16em P}}
\def\F{{\rm I\kern-0.16em F}}
\def\B{{\rm I\kern-0.16em B}}
\def\C{{\rm I\kern-0.46em C}}
\def\G{{\rm I\kern-0.50em G}}

\usepackage{url}

\begin{document}
\mainmatter              
\title{Global universality of the expected number \\of zeros of non-analytic random signals}
\titlerunning{Universality for zeros of non-analytic random signals}  
%
\author{Jürgen Angst \and Thibault Pautrel \and Guillaume Poly}
\authorrunning{Jürgen Angst, Thibault Pautrel, Guillaume Poly.} 
%
\tocauthor{Jürgen Angst, Thibault Pautrel Guillaume Poly}
\institute{Univ Rennes, CNRS, IRMAR - UMR 6625, F-35000 Rennes, France.\\
\email{jurgen.angst@univ-rennes.fr}, \email{thibault.pautrel@univ-rennes.fr}, \email{guillaume.poly@univ-rennes.fr}
}

\maketitle              

\begin{abstract}
We study the asymptotics as $n$ goes to infinity of $\mathbb E\left[\mathcal{N}(S_n,[0,2\pi])\right]$, the expected number of zeros in $[0, 2\pi]$  of a random periodic signal $S_n$ of the form 
\[
S_n(t)=\sum_{k=1}^{n}a_k f(kt),
\]
where $f$ is a non-analytic $2\pi-$periodic function and the coefficients $(a_k)$ are i.i.d. random variables, centered with unit variance. We show in particular that if $a_1$ admits a finite third moment and if the function $f$ is piecewise polynomials and of class $\mathcal C^{7}$, then we have the following universal asymptotics, independent of the particular law of the coefficients $(a_k)$
\[
\lim_{n \to +\infty}\frac{\Esp\left[\mathcal{N}(S_n,[0,2\pi])\right]}{n}= \frac{2}{\sqrt{3}}\sqrt{\frac{\|f'\|_{L^2([0,2\pi])}}{\|f\|_{L^2([0,2\pi])}}}.
\]
This result thus extends in expectation and at the scale of the whole period $[0,2\pi]$ the local universality property  established {in [Angst-Poly, IMRN, 2019]}, in distribution and in shrinking intervals of size $1/n$.  Moreover, it generalizes to a non-analytic context the global universality results obtained in the more classical frameworks of random trigonometric polynomials or random analytic functions. Our approach combines a new almost sure Central Limit Theorem \`a la Salem--Zygmund for the function $S_n$ when evaluated at a uniform random point in $[0, 2\pi]$, and as well as suitable uniform integrability and anti-concentration estimates.

\end{abstract}

\section{Introduction and statement of the results}

\subsection{Introduction}
The study of the nodal sets or level sets of random functions is the object of vast literature and connects various domains of mathematics, from probability theory to algebraic geometry and Riemannian geometry, or else mathematical physics. Of particular interest is the study of the zero sets of random algebraic polynomials,  random trigonometric polynomials, or random combinations of Laplace eigenfunctions, and more specifically the universality of their asymptotics in the large degree / high energy regimes. 
Let us briefly recall what the notion of universality means in these contexts.
Let $(a_k)_{k\geq 1}$ be a sequence of random variables, centered with unit variance, defined on a probability space $(\Omega,\mathcal{F},\Prob)$ and let $(\varphi_k)_{k\geq 1}$ be a family of real functions. One then consider the random linear combinations
$$S_n(t) :=\sum_{k=1}^{n}a_k \varphi_k(t).$$
If $\varphi_k(t)=t^{k}$ or $\varphi_k(t) =\cos(kt)$ or $\sin(kt)$, one recovers for example the classical models of Kac random algebraic polynomials and trigonometric polynomials. We are then interested in the number of zeros of $S_n$ on a given interval, denoted by
\[
\mathcal{N}(S_n,[a,b]) :=\text{Card}\{t\in[a,b], S_n(t)=0\}.
\]
The question of the universality of the asymptotics of $\mathcal{N}(S_n,[a,b])$ as $n$ goes to infinity consists in asking whether the latter does depend or not on the particular law of the random coefficients $(a_k)$ and their correlations. Since the observable of interest $\mathcal{N}(S_n,[a,b])$ is random, the question of the universality can be considered in various probabilistic ways, e.g. in distribution, in expectation or even almost surely. Besides, one can study the local universality properties at a microscopic scale, e.g. in an interval $[a_n, b_n]$ whose length goes to zero as $n$ goes to infinity, or at a macroscopic scale, in which case one then speaks of global universality properties. 
\par
\medskip
In both algebraic and trigonometric frameworks, the literature dealing with universality properties of nodal sets associated with random polynomials is thriving. The reader can refer for example to the articles \cite{TV15,IKM16,NV18} and the references therein for local properties and to \cite{AP15,Fla17,ADP19,AP21,Mat12,DNV18,Muk19,ANP24} for universality results at a global  
scale. 
In the opposite direction, there are also examples of models where the asymptotics of the number of zeros does depend on the nature of the coefficients $(a_k)$, see for example \cite{Pir19a,Pir19b,Pau20,APP22}. Interestingly, there are some models where the first order asymptotics is universal whereas the second order asymptotics is not, see \cite{BCP18,DNN20,AP23}. 
\par
\medskip
For the most part of the existing literature about universality properties, the analyticity of the random functions $\varphi_k$ -- be they algebraic polynomials, trigonometric functions or else -- plays a crucial role in the corresponding results.  For example, Hurwitz Theorem ensures the continuity of the nodal sets in \cite{IKM16}, and Jensen integral formula is the starting point of \cite{NV18} to count the number of zeros. This raises naturally the role played by the analyticity of the involved functions in the universality properties of nodal domains. If we have in mind the isolated zeros principle, it seems reasonable to think that the regularity of the function could affect the distribution of zeros.
\par
\medskip
The goal of this article, which continues the study initiated in \cite{AP20IMRN}, is to show that the analyticity hypothesis is in fact non-necessary to obtain universal asymptotics, at both local and global scales. The method we use to derive the global asymptotics of the expected number of zeros indeed requires only a finite order regularity and can be applied for example to piecewise polynomial functions. It turns out that the main necessary ingredients are some a priori bounds on the number of zeros of the considered signals (which are free when considering algebraic or trigonometric polynomials) and not surprisingly, some anti-concentration estimates on the size of the signal. It is precisely the latter anti-concentration estimates that require the most regularity and it would be interesting / challenging to derive similar bounds with minimal regularity.

\subsection{The model and existing results}

Let us specify the model of random periodic signals we consider here. For the most part, we work under the same framework as in the local universality study initiated in \cite{AP20IMRN}.
The model thus consists in choosing random linear combinations of functions $\varphi_k$ of the form $\varphi_k(t)=f(kt)$ where the base function $f$ is a continuous, $2\pi$-periodic. The typical and ``ideal'' base function we have in mind is for example the triangular function interpolating the cosine function on $ \pi \mathbb Z$, namely 
\[
f(x)=\left \lbrace \begin{array}{ll} 1+\frac{2}{\pi} x & \text{on}\;\; [-\pi, 0],\\
1-\frac{2}{\pi} x & \text{on}\;\; [0, \pi]. \end{array} \right.
\]
To avoid pathological examples, we will suppose in the sequel that the function $f$ satisfies the following condition
\begin{description}
\item[(H1)] $\begin{array}{l} \qquad  f \in H^1, \;\; \text{i.e.}\;\;  ||f||_2+||f'||_2<+\infty, \;\; \text{with}\;\; \langle f,f \rangle > 0 ~\text{and}~\langle f',f'\rangle >0, \\
\\ \qquad {f \;\; \text{is centered i.e.} \;\; \langle f,1 \rangle =\int_0^{2\pi} f(x)dx =0},
\end{array} $
\end{description}
where $\langle\cdot,\cdot\rangle$ denotes the usual $L^2([0,2\pi])$ scalar product and $\|\cdot\|_2$ is the associated norm.
To make appear converging quantities as $n$ goes to infinity, we will in fact consider the rescaled process
$$f_n(t):=\frac{1}{\sqrt{n}}\sum_{k=1}^{n}a_kf(kt),$$
where {\bf $(a_k)_{k\geq 1}$ is a sequence of i.i.d. random variables}, centered  with variance one and defined on a common probability space $(\Omega,\mathcal{F},\Prob)$. Naturally the nodal set of $f_n$ coincides with the one of the original signal $S_n$.

\begin{figure}[ht]
\begin{center}
\begin{minipage}{12.6cm}
{\scriptsize
\includegraphics[scale=0.35]{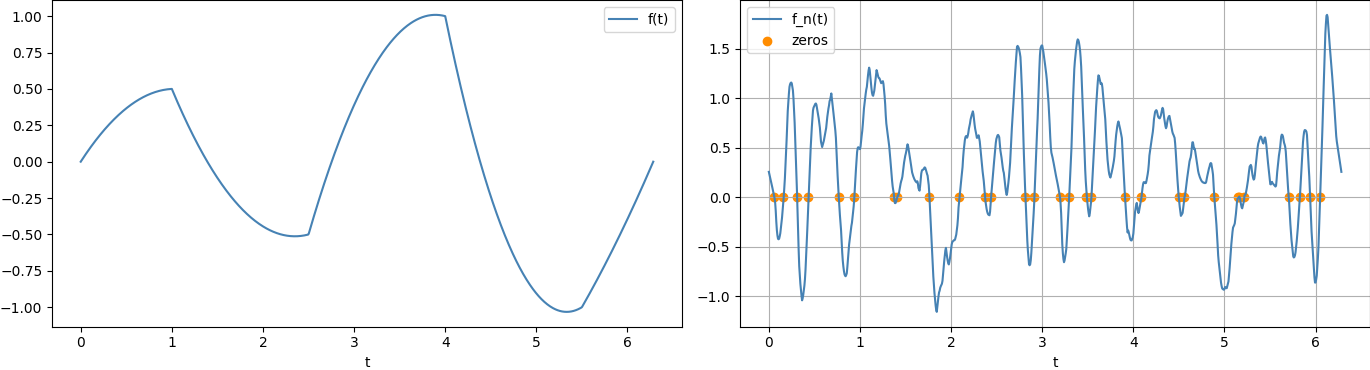}
\caption{{An example of a  piecewise polynomial function $f$ and a sample of an associated random signal $f_n$}.}
}
\end{minipage}
\end{center}
\end{figure}

The observable of interest here is the number of zeros of $f_n$ on the whole period $[0,2\pi]$, and more specifically its expectation $\mathbb E$ with respect to the law $\mathbb P$ of the coefficients
\[
\mathbb E\left[\mathcal{N}(S_n,[0,2\pi])\right]= \mathbb E\left[\mathcal{N}(f_n,[0,2\pi])\right].
\]
As mentioned above, a first local universality result was obtained in \cite{AP20IMRN}, namely a convergence  distribution and in shrinking intervals of size $1/n$. To be more precise, let $\alpha \in (0,2\pi) \backslash \pi\mathbb Q$ and let $(p_n)$ be a sequence of real numbers such that $p_n/n$ converges to $\alpha$ as $n$ goes to infinity. The authors considered the localized process 
\[
X_n(t)  :=\displaystyle{\frac{1}{\sqrt{n}} \sum_{k=1}^n a_k   f\left( \frac{k(p_n+t)}{n} \right)}=f_n\left( \frac{p_n+t}{n}\right), \quad t \in [0,2\pi].
\]
Under various suitable conditions on the base function $f$ and the common law of the coefficients $(a_k)$, they  established that this process has a non-trivial scaling limit and then deduced the convergence in distribution of the associated number of zeros.

\begin{theorem}[Combination of the main results of \cite{AP20IMRN}]\label{thm.IMRN}
Suppose that the function $f$ satisfies the condition $\mathbf{(H1)}$. {Then, there exists a stationary and centered Gaussian process $(X_{\infty}(t))_{t \in [0, 2\pi]}$} such that as $n$ goes to infinity, the sequence stochastic processes $(X_n(t))_{t \in [0,2\pi]}$
converges to $(X_{\infty}(t))_{t \in [0, 2\pi]}$ in the sense of finite marginals.
Moreover, under each of the following additional conditions,
\begin{enumerate}
\item the function $f$ is $\mathcal C^2$,
\item there exists $\beta>2$ and $\gamma>0$ with $\beta \gamma  >1$  such that $f$ is $\mathcal C^{1+\gamma}$ and $\mathbb E[|a_1|^{\beta}]<+\infty$,
\item there exists $\gamma>0$ such that $f$ is $\mathcal C^{1+\gamma}$ and $\mathbb E[|a_1|^4]<+\infty$,
\item $\alpha$ is Diophantine and $p_n/n$ converges to $\alpha$ at a polynomial rate, $f$ is piecewise linear and $\mathbb E[|a_1|^4]<+\infty$,
\end{enumerate}
\smallskip
we have the following convergence in distribution as $n$ goes to infinity
\[
\mathcal N\left( f_n, \left[ \frac{p_n}{n}, \frac{p_n+2\pi}{n}  \right]\right) =\mathcal N\left( X_{n} , \left[ 0, 2\pi \right]\right) \xrightarrow{law} \mathcal N\left( X_{\infty} , \left[ 0, 2\pi \right]\right).
\]
Moreover, the limit process is such that 
\[
\mathbb E\left[ \mathcal N\left( X_{\infty} , \left[ 0, 2\pi \right]\right)\right] = \frac{2}{\sqrt{3}} \sqrt{\frac{\langle f', f'\rangle}{\langle f, f\rangle}}.
\]
\end{theorem}

The key step in the above statement is the fact that, under each of the above four conditions, the distributions of the processes $(X_n(t))$ are tight in good function space topologies, for example the $\mathcal C^1-$topology or the Lipschitz topology. This ensures the continuity of the associated nodal sets and thus the convergence of the number of zeros. The less regular is the base function $f$, the more difficult and technical is the proof of the tightness, which then requires some additional integrability assumptions on the law of the coefficients.

\subsection{Main results and comments} \label{sec.mainresults}

Let us now describe the main results of the paper, whose detailed proofs are postponed in the next sections. As announced, we will establish the global universality of the asymptotics of the expected number of zeros of $f_n$ over the whole period, see Theorem \ref{thm.mean.num} just below. The overall strategy extends the one adopted in the recent contributions \cite{AP21,APP22,AP23}. It relies on variations on the celebrated almost sure Central Limit Theorem by Salem and Zygmund for trigonometric polynomials obtained in the seminal paper \cite{SZ54}.
\par
\medskip
To be more precise, let $X$ be a random variable whose distribution $\mathbb P_X$ is uniform on the interval $[0,2\pi]$ and which is independent of the entries $(a_k)_{k\geq 1}$. In other words, we work on the product space
$$\left(\Omega \times [0,2\pi], \mathcal{F} \times \mathcal{B}([0,2\pi]),\Prob\otimes \Prob_X\right),$$
with $X$ seen as the second projection on $([0,2\pi],\mathcal{B}([0,2\pi]))$.
We denoted by $\Esp_{X}$ the expectation with respect to the probability $\Prob_X$. 
We first show that almost surely with respect to $\mathbb P$, the sequence $f_n$ evaluated at the random variable $X$ converges in distribution under $\mathbb P_X$ to a Gaussian variable. 

\begin{theorem}\label{thm.salem.func.gen}
Suppose that the base function $f$ satisfies the condition $\mathbf{(H1)}$ and that the common law of the $(a_k)$ admits a finite third moment. Then, $\Prob$-almost surely, for all $t \in \R$, we have
\[
\lim_{n\to +\infty} \Esp_X\left[e^{itf_n(X)}\right]= e^{-\frac{t^2}{2}\langle f,f \rangle}.
\]
In other words, $\Prob-$almost surely, the sequence of random variables $(f_n(X))_{n\geq 0}$ converges in distribution under $\Prob_X$ towards a {centered} Gaussian variable {with variance $\langle f,f \rangle$}.
\end{theorem}

As in the case of random trigonometric polynomials considered in \cite{AP21}, the above almost sure Central Limit Theorem can be extended to a functional Central Limit Theorem. 
Namely, if we introduce the stochastic process $(g_{n}(t))_{t\in[0,2\pi]}$ defined by
$$g_n(t):=f_n\left(X+\frac{t}{n}\right),~~t\in [0,2\pi],$$
we have the following convergence.

\begin{theorem}\label{thm.salem.func.gen.functional}Suppose that the base function $f$ satisfies the condition $\mathbf{(H1)}$ and that the common law of the $(a_k)$ admits a finite third moment. Then, $\Prob-$almost surely, as $n$ goes to infinity, the finite dimensional marginals of process $(g_{n}(t))_{t\in[0,2\pi]}$ converge in distribution  to the ones of the stationary Gaussian process $(g_{\infty}(t))_{t\in[0,2\pi]}$ with covariance function
$$\Esp_X[g_{\infty}(t)g_{\infty}(s)] = \rho(t-s),~~\text{with}~\left\{\begin{array}{l}\displaystyle \rho(u):=\frac{1}{u}\int_0^{u}(f\ast \widetilde{f})(x)dx,~u\neq 0 \\\rho(0)=\langle f,f \rangle\end{array}\right.,$$
where $\widetilde{f}(x):=f(-x)$ for all $x\in \R$. Moreover, if the function $f$ is of class $\mathcal C^2$, the process $(g_{n}(t))_{t\in[0,2\pi]}$ converges is distribution in the $\mathcal{C}^1$ topology towards $(g_{\infty}(t))_{t\in[0,2\pi]}$.
\end{theorem}

\begin{remark}
The law of the limit process $(g_{\infty}(t))$ appearing in Theorem \ref{thm.salem.func.gen.functional} actually coincides with the one of the process $(X_{\infty}(t))$ in Theorem \ref{thm.IMRN} above, when looking at the signal $S_n$ in a deterministic (but possibly moving) window of size $1/n$. As observed in \cite{AP20IMRN}, this limit process is the natural generalization to our non-analytic context of the standard Gaussian process with $\sin_c$ covariance function appearing as a universal limit in many models, in particular in the case of random trigonometric polynomials. 
\end{remark}

\begin{remark}
Note that, in the last statement, the $\mathcal C^2$ regularity condition ensuring the tightness of the distributions of $(g_{n}(t))_{t\in[0,2\pi]}$ in the $\mathcal C^1$ topology is not necessary. As stated in Theorem \ref{thm.IMRN} above, it could be replaced by weaker regularity assumptions on $f$, provided the random coefficients $(a_k)$ are sufficiently integrable. 
\end{remark}

\begin{remark}
Since the limit process $(g_{\infty}(t))$ is non-degenerate (see Lemma \ref{lm.nd.kac} below), under the assumptions of Theorem \ref{thm.salem.func.gen.functional}, the continuous mapping Theorem ensures the convergence in distribution under $\mathbb P_X$ (and $\mathbb P$ almost-surely)
\[ 
\mathcal N\left( f_{n} , \left[ X, X+\frac{2\pi}{n} \right]\right)=\mathcal N\left( g_{n} , \left[ 0, 2\pi \right]\right) \xrightarrow{law} \mathcal N\left( g_{\infty} , \left[ 0, 2\pi \right]\right),
\]
which is, when compared to Theorem \ref{thm.IMRN}, another local universality principle at scale $1/n$, this time in a random window based at the uniform point $X$.
\end{remark}
In order to pass from the above local results to the global scale of the whole period, we have to impose more conditions on the base function $f$. When dealing with algebraic or trigonometric polynomials, we have naturally some a priori bounds on the number of zeros as a function of the degree : it grows at most linearly. For a general base function $f$, this is of course not the case and this is the reason why we will restrict to piecewise polynomial functions for which the number of zeros is at most quadratic in the degree, see Lemma  \ref{lm.apriori} below. We thus introduce the following condition.
\begin{description}
\item[(H2)]
\;\;  $f$ is piecewise polynomial, with a global regularity of class $\mathcal{C}^{7}$ on the whole period, and for each $i\in \{0,\dots, 6\}$, there exists $x_i \in [0,2\pi]$ such that $f^{(i)}(x_i)\neq 0$. \end{description}

\begin{remark}
The non-zero condition for $f$ and its derivatives featured in $\mathbf{(H2)}$ is of course satisfied in the trigonometric case where $f(x)=\cos(x)$ or $f(x)=\sin(x)$. The global regularity exponent 7 is not meant to be sharp here, it could be lowered, but at the cost of making the proofs of anti-concentration and uniform integrability estimates presented in Section 3 much more technical.
\end{remark}
\par
\medskip
As announced above, under these additional conditions, the main result of the article is the following global universality principle.

\begin{theorem} \label{thm.mean.num} Suppose that the base function $f$ satisfies conditions $\mathbf{(H1)}, \mathbf{(H2)}$, and that the common law of the $(a_k)$ admits a finite third moment. Then, as $n$ goes to infinity, we have
$$\lim_{n\to +\infty} \frac{\Esp\left[\mathcal{N}(f_n,[0,2\pi])\right]}{n} = \frac{2}{\sqrt{3}} \sqrt{\frac{\langle f', f'\rangle}{\langle f, f\rangle}}.$$
\end{theorem}

\begin{remark} The previous Theorem \ref{thm.mean.num} calls for several remarks.
\begin{enumerate}
\item If the generic function $f$ is cosine or sine, we have $\langle f, f\rangle = \langle f',f'\rangle>0$ and we recover the well-known fact, first established in \cite{Dun66}, that when properly normalized, the expected number of zeros over a period converges to  $\frac{2}{\sqrt{3}}$. 
\item It remains an open question to determine whether or not the random superposition of triangular signals satisfies this same limit, since in this case, we cannot rely on the regularity properties of $f$ in establishing the equi-integrability condition, crucial for global asymptotics in this method.
\end{enumerate}
\end{remark}

The plan of the rest of the article is the following. The next Section \ref{sec.uni.SZ} is dedicated to the proof of Theorems \ref{thm.salem.func.gen} and \ref{thm.salem.func.gen.functional}, i.e. the Salem--Zygmund-like almost sure CLTs for the rescaled process $(g_n(t))$ in a ``strong'' topology. In Section \ref{sec.study.nb.z}, we establish a multidimensional small ball estimate for the signal $g_n$ and its derivative (Proposition \ref{pro.kolmorogo}) and then we derive from this estimate the uniform integrability of the normalized number of zeros of $g_n$ on the whole period (Proposition \ref{prop.eq.int}). To facilitate the global reading of the paper, the most technical proofs are postponed in the last Section \ref{sec.tech}.

\section{A functional almost sure Central Limit Theorem} \label{sec.uni.SZ} 
This section is devoted to {the proofs} of both Theorem \ref{thm.salem.func.gen} and its functional analogue Theorem \ref{thm.salem.func.gen.functional}, i.e. the $\mathbb P-$almost sure Central Limit Theorems associated with the original signal $S_n$ when evaluated at an independent and uniform random point.

\subsection{A Central Limit Theorem of $f_n(X)$} \label{sec.dim1}
The strategy of proof of Theorem \ref{thm.salem.func.gen} follows globally the same lines as the original one given by Salem and Zygmund in \cite{SZ54} in the case trigonometric polynomials with random signs. 
More specifically, we first establish an $L^2(\Prob)$ estimate, then we use a Borel--Cantelli argument to obtain the desired almost sure convergence along a subsequence and conclude by getting rid of taking the subsequence.
Recall that the goal is to establish that, $\mathbb P$ almost surely, for all ${\lambda} \in \mathbb R$, we have 
\[
\lim_{n \to +\infty} \Esp_X\left[e^{i {\lambda} f_n(X)} \right]= e^{-\frac{{\lambda}^2}{2} \langle f,f\rangle}.
\]
Let us set
\[
\sigma_n^{2}(X):=\Esp\left[f_n^{2}(X)\right]=\frac{1}{n} \sum_{k=1}^n f^2(k X).
\]
By Birkhoff ergodic Theorem, almost surely under $\mathbb P_X$, we have 
\[
\lim_{n \to+\infty} \sigma_n^{2}(X):=\langle f,f\rangle,
\]
so that by dominated convergence, Theorem \ref{thm.salem.func.gen} amounts to show that $\mathbb P-$almost surely, for all ${\lambda} \in \mathbb R$
\[
\lim_{n \to +\infty} \left|\Esp_X\left[e^{i {\lambda} f_n(X)}-e^{-\frac{{\lambda}^2}{2}\sigma_n^2(X)]}\right]\right|=0.
\]
As announced, we will first establish an $L^2(\mathbb P)$ estimate by considering
\[
\Delta_n({\lambda}):=\Esp\left[ \left|\Esp_X\left[e^{it {\lambda}f_n(X)}-e^{-\frac{{\lambda}^2}{2}\sigma_n^2(X)]}\right]\right|^2\right].
\]
\begin{lemma}\label{lem.delta}
If $a_1$ admits a finite third moment and if $f$ satisfies the condition $\mathbf{(H1)}$, there exists a positive constant $C$ such that 
\[
\Delta_n({\lambda}) \leq \frac{C \,({\lambda}^2+ |{\lambda}|^3)}{\sqrt{n}}.
\]
\end{lemma}
To facilitate the global reading of the paper, the proof of Lemma \ref{lem.delta} is postponed in Section \ref{sec.appdelta} below. 
Let us then fix an integer $\gamma>2$ and ${\lambda} \in \mathbb R$. By Borel--Cantelli Lemma, $\Prob-$almost surely (i.e. on a set of full measure which depends on ${\lambda}$), as $n$ goes to infinity, Lemma \ref{lem.delta} ensures that along the subsequence $n^{\gamma}$
\begin{equation}\label{eq.convpsBC}
\lim_{n\to+\infty}\left| \Esp_X\left[e^{i {\lambda} f_{n^{\gamma}}(X)}-e^{-\frac{{\lambda}^2}{2}\sigma_{n^{\gamma}}^{2}(X)} \right]\right| =0.
\end{equation}
Now, let $m$ a given integer. There exists a unique $n$ such that $n^{\gamma }< m \leq (n+1)^{\gamma}$. Then, using the fact that the exponential is Lipschitz and  Cauchy--Schwarz inequality, we have
\[
\left| \Esp_X\left[e^{i {\lambda} f_{n^\gamma}(X)}\right]-\Esp_X\left[e^{i  {\lambda} f_m(X)}\right]\right|^2 \leq  {\lambda}^2 \Esp_X \left[|f_{n^\gamma}(X)-f_m(X)|^2\right].
\]
Expliciting the square, we get 
\[
\begin{array}{ll}
\displaystyle{\Esp_X \left[(f_{n^\gamma}(X)-f_n(X))^2\right]} &\leq \displaystyle{ 2\left(1-\sqrt{\frac{n^\gamma}{m}} \right)^2\Esp_X\left[f_{n^\gamma}(X)^2\right]}\\
\\
& \displaystyle{+2\left( 1 - \frac{n^{\gamma}}{m}\right)\Esp_X\left[\left(\frac{1}{\sqrt{m-n^{\gamma}}} \sum_{k=n^\gamma+1}^{m}a_kf(kX)\right)^2\right].}
\end{array}
\]
We have then the following Lemma,  whose proof is given in Section \ref{sec.appcontrolL2}. 
\begin{lemma} \label{lem.controlL2}
Suppose $f$ satisfies the condition $\mathbf{(H1)}$. There exists a constant $C=C(\omega)$ such that $\mathbb P-$almost surely 
\[
\sup_{n \geq 1} \mathbb E_X [f_n(X)^2] \leq C, \;\; \textrm{and} \quad \sup_{\substack{n \geq 1\\ M \geq 1}} \mathbb E_X \left[\left(\frac{1}{\sqrt{n}} \sum_{M\leq k \leq M+n}a_k f_k(X)\right)^2\right] \leq C.
\]
\end{lemma}
\noindent
By Lemma \ref{lem.controlL2}, we deduce that 
\[
\begin{array}{ll}
\displaystyle \left| \Esp_X\left[e^{i {\lambda} f_{n^\gamma}(X)}\right]-\Esp_X\left[e^{i {\lambda} f_m(X)}\right]\right|^2 & \displaystyle \leq 2C {\lambda}^2\left[ \left( 1 - \frac{n^{\gamma}}{m}\right) +\left(1-\sqrt{\frac{n^\gamma}{m}} \right)^2 \right] \\
\\
& \displaystyle=O\left(\frac{1}{m^{1/\gamma}}\right).
\end{array}
\]
Combining this last estimate with Equation \eqref{eq.convpsBC}, one can indeed conclude that for a fixed ${\lambda}$, $\mathbb P-$almost surely 
\[
\lim_{m \to +\infty} \Esp_X\left[e^{i {\lambda} f_m(X)} \right]= \lim_{m \to +\infty} \mathbb E_X \left[ e^{-\frac{{\lambda}^2}{2} \sigma_m^2(X)} \right]= e^{-\frac{{\lambda}^2}{2} \langle f,f\rangle}.
\] 
Therefore, we deduce that the above convergence holds $\mathbb P-$almost surely for all ${\lambda} \in \mathbb Q$. To conclude that the asymptotics is valid $\mathbb P-$almost surely for all ${\lambda},{\lambda'} \in \mathbb R$, we remark that 
\[
\left| \Esp_X\left[e^{i {\lambda} f_{n}(X)}\right] -  \Esp_X\left[e^{i {\lambda'} f_{n}(X)}\right] \right|^2 \leq |{\lambda}-{\lambda'}|^2  \mathbb E_X [f_n(X)^2], 
\]
so that in virtue of Lemma \ref{lem.controlL2} above, $\mathbb P-$almost surely 
\[
\sup_{|{\lambda}-{\lambda'}|\leq \varepsilon}\sup_{n \geq 1} \left| \Esp_X\left[e^{i {\lambda} f_{n}(X)}\right] -  \Esp_X\left[e^{i {\lambda'} f_{n}(X)}\right] \right|^2 \leq C \, \varepsilon^2. 
\]

\subsection{A functional Central Limit Theorem for $(g_n(t))$}\label{sec.func.SZ}
In this section, we give a detailed proof of Theorem \ref{thm.salem.func.gen.functional}, which is the functional analogue of the previous Theorem \ref{thm.salem.func.gen}. Recall that
\[
g_n(t):=f_n\left(X+\frac{t}{n}\right)=\frac{1}{\sqrt{n}}\sum_{k=1}^{n}a_kf\left(kX+\frac{kt}{n}\right), \quad t \in [0, 2\pi].
\]
The goal is to establish that $\Prob-$almost surely, the family of processes $(g_n(t))_{t\in [0,2\pi]}$ converge in distribution, under $\Prob_X$ and with respect to the $\mathcal C^1-$topology, towards an explicit stationary Gaussian process. As done classically, we proceed in two steps: first, we show the convergence of the finite dimensional marginals. Then, we focus on the tightness with respect to the $\mathcal C^1-$topology, which will be established using a Lamperti-like criterion.

\subsubsection{Convergence of finite dimensional marginals}
Let us first establish the convergence of the finite dimensional marginals of the localized process $(g_n(t))_{t\in[0,2\pi]}$. 
\begin{proposition} \label{prop.finite.marg}
If $a_1$ admits a finite moment of order three and if $f$ satisfies the condition $\mathbf{(H1)}$, then $\Prob-$almost surely, the finite marginals of $(g_n(t))_{t\in[0,2\pi]}$ converge to the ones of a stationary Gaussian process $(g_{\infty}(t))_{t\in [0,2\pi]}$ with covariance function
$$\Esp_X[g_{\infty}(t)g_{\infty}(s)] = \rho(t-s),~~\text{with}~\left\{\begin{array}{l}\displaystyle \rho(u):=\frac{1}{u}\int_0^{u}(f\ast \widetilde{f})(x)dx,~u\neq 0 \\\rho(0)=\langle f,f \rangle \end{array}\right.,$$
where $\widetilde{f}(x):=f(-x)$ for $x\in \R$.
\end{proposition}
\noindent
For a fixed positive integer $M$ and $t=(t_1,\dots, t_M) \in [0,2\pi]^{M}$,  $\lambda=(\lambda_1,\dots,\lambda_M)\in \R^M$, we set
\[
Z_n(X,t,\lambda):=\sum_{p=1}^{M}\lambda_pg_n(t_p) =\sum_{p=1}^{M}\lambda_p f_n\left(X+\frac{t_p}{n}\right).
\]
Establishing Proposition \ref{prop.finite.marg} comes down to proving that $\Prob-$almost surely, for all $M>1$, for all $t \in [0,2\pi]^{M}$ and $\lambda\in \R^M$, we have the following convergence of characteristic functions
\[
\lim_{n \to +\infty} \Esp_X \left[e^{iZ_n(X,t,\lambda)} \right]=\exp\left(-\frac{1}{2}\sum_{p,q=1}^{M} \lambda_p \lambda_q\rho(t_p-t_q) \right).
\]
The first reduction consists in remarking that, as in the one-dimensional case treated above, it is sufficient to establish the almost sure convergence for  fixed $t \in [0,2\pi]^{M}$ and $\lambda\in \R^M$. Indeed, if $t,s \in [0,2\pi]^{M}$ and $\lambda, \mu \in \R^M$ we have the following Lemma, whose proof is given in Section \ref{sec.first reduc}. 
\begin{lemma}\label{lem.first reduc}Under the hypotheses of Proposition \ref{prop.finite.marg}, $\mathbb P-$almost surely there exists a positive constant $C=C(\omega)$ such that for any $0<\varepsilon<1$, $0<\eta<1$
\[
\sup_{\substack{||t-s||_{\infty} \leq \varepsilon \\ ||\lambda-\mu||_1 \leq \eta\\ n \geq 1}}\left| \Esp_X \left[e^{iZ_n(X,t,\lambda)} \right] -\Esp_X \left[e^{iZ_n(X,s,\mu)} \right]  \right|\leq C\left( \eta+||\lambda||_1 \times \varepsilon\right).
\]
\end{lemma} 
The almost sure uniform continuity ensured by Lemma \ref{lem.first reduc} above thus allows to restrict to the case where $t \in ([0,2\pi]\cap \mathbb Q)^{M}$ and $\lambda \in \mathbb Q^M$ and therefore to fixed values of $t$ and $\lambda$. 
The next reduction comes from the following Lemma, which allows to identify the limit covariance and whose proof is also given in Section \ref{sec.varZ}.
\begin{lemma}\label{lem.varZ}Under the hypotheses of Proposition \ref{prop.finite.marg}, we have
\begin{equation} \label{eq.lim.func} \lim_{n \to +\infty}\Esp_X\left[e^{-\frac{1}{2}\Esp[Z_n(X,t,\lambda)^2]}\right]=\exp\left(-\frac{1}{2}\sum_{p,q=1}^{M} \lambda_p \lambda_q\rho(t_p-t_q) \right).
\end{equation}
\end{lemma} 

Thanks to Lemma \ref{lem.varZ}, Proposition \ref{prop.finite.marg} is thus equivalent to the following convergence, $\Prob-$almost surely, 
\[
\lim_{n \to +\infty} \left| \Esp_X \left[e^{iZ_n(X,t,\lambda)} \right]  - \Esp_X\left[e^{-\frac{1}{2}\Esp[Z_n(X,t,\lambda)^2]}\right]\right|=0.
\]
The proof then follows the same lines as its one dimensional counterpart Theorem \ref{thm.salem.func.gen}. Namely, we have first the following $L^2(\mathbb P)$ estimate, which is the analogue of Lemma \ref{lem.delta} above and whose proof is given in Section  \ref{sec.delta2} below.  
 
\begin{lemma} \label{lem.delta2}Under the hypotheses of Proposition \ref{prop.finite.marg}, there exists a constant $C$ such that $\mathbb P-$almost surely
\[
\widetilde{\Delta}_n:=\mathbb E\left[ \left| \Esp_X \left[e^{iZ_n(X,t,\lambda)} \right]  - \Esp_X\left[e^{-\frac{1}{2}\Esp[Z_n(X,t,\lambda)^2]}\right]\right|^2\right] \leq \frac{C}{\sqrt{n}}.
\]
\end{lemma}
\noindent
By Borel--Cantelli Lemma, if $\gamma>2$, one thus deduces that $\Prob-$almost surely, we have 
\[
\lim_{n\to +\infty}\left| \Esp_X\left[e^{iZ_{n^{\gamma}}(X,t,\lambda)}\right]-\Esp_X\left[e^{-\frac{1}{2}\Esp[Z_{n^{\gamma}}(X,t,\lambda)^2]}\right]\right|=0.
\]
As above, one finally gets rid of the subsequence $n^{\gamma}$ using the following estimates, whose proof is again postponed in Section \ref{sec.ssuite.est}.
\begin{lemma}\label{lma.ssuite.est}Under the hypotheses of Proposition \ref{prop.finite.marg}, as $m$ goes to infinity, if $n$ is the unique integer such that  $n^{\gamma}<m\leq (n+1)^\gamma$, then $\Prob-$almost surely,  we have 
\[
\left|\Esp_X \left[e^{iZ_{n^\gamma}(X,t,\lambda)}\right]-\Esp_X\left[e^{iZ_m(X,t,\lambda)}\right]\right| =O \left(\left( 1-\frac{n^{\gamma}}{m}\right)\right)=O\left(\frac{1}{m^{1/\gamma}}\right).
\]
\end{lemma}


\subsubsection{Tightness under regularity assumptions} \label{sec.tightness}
Having established the convergence of finite marginals of the process $g_n$, we focus now on the tightness of the family of processes $(g_n)_{n \geq 1}$ in a suitable topology. In the next applications to nodal asymptotics, we will work with the $\mathcal{C}^1-$topology so that it is reasonable to impose some minimal regularity to the base function $f$.

\begin{proposition}\label{prop.tight}
Suppose that the function $f$ is of class $\mathcal C^2$, then 
$\Prob-$almost surely, the family of distributions of $(g_n(t))_{t\in[0,2\pi]}$ under $\Prob_X$ is tight with respect to the $\mathcal{C}^1$ topology. 
\end{proposition}

\begin{proof}
The tightness with respect to the $\mathcal{C}^1-$topology of the family $(g_n)_{n \geq 1}$ under $\mathbb P_X$ is guaranteed by the following standard Lamperti-like criteria, see \cite{RS01}
\begin{equation}
\exists C>0, \quad \mathbb E_X[ |g_n(t)-g_n(s)|^2] \leq C |t-s|^2, \quad  \mathbb E_X[ |g_n'(t)-g_n'(s)|^2] \leq C |t-s|^2.
\end{equation}
The above Proposition \ref{prop.tight} is therefore a direct consequence of the strong law of large numbers, combined with the following Lemma \ref{lem.lamperti}, whose proof is given in Section \ref{sec.appdelta}. 
\qed
\end{proof}

\begin{lemma}\label{lem.lamperti}
Under the hypotheses of Proposition \ref{prop.tight}, for all integer $n \geq 1$ and real numbers $t,s$, we have
\[
 \mathbb E_X[ |g_n(t)-g_n(s)|^2] \leq 12 \times ||f'||_2^2  \left(  \frac{1}{n} \sum_{k=1}^n a_k^2 \right) |t-s|^2,
 \]
 and similarly 
 \[
 \mathbb E_X[ |g_n'(t)-g_n'(s)|^2] \leq 12 \times ||f''||_2^2  \left(  \frac{1}{n} \sum_{k=1}^n a_k^2 \right) |t-s|^2.
 \]
\end{lemma}

\subsubsection{Non-degeneracy of the limit process}
We now establish that the process $g_{\infty}$ is non-degenerate in the sense that $\Prob \otimes \Prob_X-$almost surely, we have
$$(g_{\infty}(t)=0) \Longrightarrow (g_{\infty}'(t)\neq 0).$$
For the sake of self-containedness, we recall the following lemma, first stated in \cite{AP20IMRN}. 
\begin{lemma}\label{lm.nd.kac}
The stationary Gaussian process $(g_{\infty}(t))_{t\in[0,2\pi]}$ is $\Prob \otimes \Prob_X-$almost surely non-degenerate, i.e. $\Prob \otimes \Prob_X-$almost surely, we have
$$\inf_{t\in [0,2\pi]}(|g_{\infty}(t)| +|g'_{\infty}(t)|) >0,$$
and if $[a,b]\subset [0,2\pi]$, its expected number of zeros is
$$\Esp_{\Prob \otimes \Prob_X}\left[\mathcal{N}(g_{\infty},[a,b])\right] =\frac{b-a}{\pi}\sqrt{\frac{\langle f',f'\rangle}{3\langle f,f\rangle}}.$$
\end{lemma}
\begin{proof}
Using the periodicity of $f$ and the independence of the random coefficients $(a_k)$, an immediate computation gives for any $t \in [0,2\pi]$
\[
\Esp_{\Prob\otimes \Prob_X}[g_n(t)^2]=\langle f,f\rangle, \qquad \Esp_{\Prob\otimes \Prob_X}[g_n(t)g_n'(t)]  =0,
\]
and 
\[
\Esp_{\Prob\otimes \Prob_X}[g_n'(t)^2] = \left( \frac{1}{n} \sum_{k=1}^n \frac{k^2}{n^2} \right) \langle f',f'\rangle\xrightarrow[n \to +\infty]{} \frac{\langle f',f'\rangle}{3}.
\]
Therefore the law of $(g_{\infty}(t),g_{\infty}'(t))$ does not depend on $t$ and is the one of a centered Gaussian vector with covariance matrix
\[
\Gamma:=\left(\begin{array}{cc}\langle f,f\rangle& 0 \\0 & \frac{1}{3}\langle f',f'\rangle \end{array}\right).
\]
Since we have assumed that $\langle f,f\rangle >0$ and $\langle f',f'\rangle >0$, we deduce that $\det \Gamma >0$. In particular, the Gaussian vector $(g_{\infty}(t),g_{\infty}'(t))$ admits a uniformly bounded density, uniformly in $t \in [0,2\pi]$. 
Using the classical Bulinskaya lemma, see e.g. Proposition 6.11 of \cite{AW09},  we deduce the non-degeneracy of the process.
In particular, we can apply Kac--Rice formula to compute
\[
\Esp_{\Prob\otimes \Prob_X}[\mathcal{N}(g_{\infty},[a,b])] =\frac{1}{\pi} \int_{a}^{b}\sqrt{\frac{\Esp_{\Prob\otimes \Prob_X}[g_{\infty}'(t)^2]}{\Esp_{\Prob\otimes \Prob_X}[g_{\infty}(t)^2]}}dt= \frac{b-a}{\pi}\sqrt{\frac{\langle f',f'\rangle}{3\langle f,f\rangle}}.
\]
\qed
\end{proof}


\section{Study of the number of real zeros} \label{sec.study.nb.z}
The goal of this section is to complete the proof of our main Theorem \ref{thm.mean.num} on the universality of the expected number of zeros. 
Let us first recall the following deterministic result, which allows to represent the normalized number of zeros of a periodic function as the expectation of the number of zeros in a random subinterval. The proof of this result, which a direct application of Fubini inversion of sums, can be found in \cite[Lemma 3]{AP21}.
\begin{lemma}[Lemma 3 in \cite{AP21}]
Let $X$ a random variable which is uniformly distributed over $[0,2\pi]$. Let $f$ be a continuous and $2\pi$-periodic function with a finite number of zeros over a period. Then, for any $h \in (0,2\pi)$, 
$$\frac{h}{2\pi} \times \mathcal{N}(f,[0,2\pi]) =\Esp_X\left[\mathcal{N}(f,[X,X+h])\right].$$
\end{lemma}
Taking $h=\frac{2\pi}{n}$ in the previous statement gives, for $X\sim \mathcal{U}([0,2\pi])$ independent of the entries $(a_k)_{k\geq 1}$, that
$$\frac{\mathcal{N}(f_n,[0,2\pi])}{n}=\Esp_X\left[\mathcal{N}\left(f_n,\left[X,X+\frac{2\pi}{n}\right]\right) \right].$$

Using the stochastic process $(g_n(t))_{t\in [0,2\pi]}$ defined before, the normalized number of real zeros of $f_n$ can thus be expressed as
\begin{equation} \label{eq.nb.zero.g}
\frac{\mathcal{N}(f_n,[0,2\pi])}{n} =\Esp_X\left[\mathcal{N}(g_n,[0,2\pi])\right].
\end{equation}
The main byproduct of the functional Central Limit Theorem stated in Theorem \ref{thm.salem.func.gen.functional} is the following corollary.
\begin{corollary} \label{cor.conv.nb.zeros} Suppose that the base function $f$ is of class $\mathcal{C}^2([0,2\pi])$ and satisfies the condition $\mathbf{(H1)}$, then we have
\begin{enumerate}
\item $\Prob-$almost surely, as $n$ goes to infinity and under $\mathbb P_X$, the number of zeros $\mathcal{N}(g_n,[0,2\pi])$ of the localized process $(g_n(t))_{t\in[0,2\pi]}$ converges in distribution towards $\mathcal{N}(g_{\infty},[0,2\pi])$. 
\item As $n$ goes to infinity, $\mathcal{N}(g_n,[0,2\pi])$ converges in distribution under $\Prob \otimes \Prob_X$ towards $\mathcal{N}(g_{\infty},[0,2\pi])$. 
\end{enumerate}
\end{corollary}
\begin{proof}
The first statement is direct consequence of the fact that, deterministically, the functional  ``number of roots'' is continuous with respect to  the $\mathcal{C}^1-$topology, as soon as the limit is non-degenerated, i.e.
if $u_n$ converges towards $u$ as $n\to +\infty$ for the $\mathcal{C}^1$ topology with $\inf_{t\in [0,2\pi]}(|u(t)|+|u'(t)|)>0$, then $\text{Card}\left(u_n^{-1}(\{0\})\cap [0,2\pi]\right)\to \text{Card}(u^{-1}(\{0\}) \cap [0,2\pi])$, see for example Proposition 4.1 of \cite{AP21}. In our context, thanks to Theorem \ref{thm.salem.func.gen.functional}, $\mathbb P-$almost surely and under $\mathbb P_X$, the process $\{g_n(\cdot)\}_{n\geq 1}$ converges in the $\mathcal{C}^1-$topology towards $g_{\infty}$, which is non-degenerate via Lemma \ref{lm.nd.kac}. The result thus follows from the Continuous Mapping Theorem.
For the second statement, using that if $h$ is a continuous and bounded test function, the first point gives that $\mathbb P-$almost surely, as $n$ goes to infinity,
$$\lim_{n\to+\infty} \Esp_X\left[h(\mathcal{N}(g_n,[0,2\pi]))\right]= \Esp_X\left[h(\mathcal{N}(g_{\infty},[0,2\pi]))\right].$$
Then, by dominated convergence, integrating with respect to $\mathbb P$, we obtain
\[
\lim_{n\to+\infty}\Esp_{\Prob \otimes \Prob_X} \left[h \left(\mathcal{N}(g_n,[0,2\pi])\right)\right] =\Esp_{\Prob \otimes \Prob_X}\left[h(\mathcal{N}(g_{\infty},[0,2\pi]))\right],
\]
hence the convergence in distribution under $\Prob \otimes \Prob_X$.
\qed
\end{proof}

\par \bigskip
Thanks to Equation \eqref{eq.nb.zero.g}, establishing the global universality of the expected number of zeros on the whole period  then comes down to showing that
\begin{equation}\label{eq.uni}
\lim_{n\to +\infty}\Esp_{\Prob \otimes \Prob_X} \left[\mathcal{N}(g_n,[0,2\pi])\right]=\Esp_{\Prob \otimes \Prob_X}\left[\mathcal{N}(g_{\infty},[0,2\pi])\right].
\end{equation}
In other words, we need to upgrade the convergence in law obtained in the second point of Corollary \ref{cor.conv.nb.zeros}, to a convergence of the first moment under $\mathbb P\otimes \mathbb P_X$. 
To achieve this, we will work from now on under the additional assumptions $\mathbf{(H2)}$ stated in Section \ref{sec.mainresults}.
The fact that base function $f$ is piecewise polynomial will give us an a priori upper bound for the number of zeros of $f_n$ on $[0,2\pi]$, see Lemma \ref{lm.apriori} below. Otherwise, the global regularity and the non-vanishing derivatives hypotheses will allow us to show some anti-concentration estimate, from which we will deduce that the sequence of random variables $(\mathcal{N}(g_n,[0,2\pi]))_{n \geq 1}$ is in fact bounded in $\mathbb L^p(\mathbb P \otimes \mathbb P_X)$ for some $p>1$, see Proposition \ref{prop.eq.int} below, hence ensuring the convergence \eqref{eq.uni} and the global universality statement.

\subsection{Small ball estimate}
In this section, we derive some anti-concentration estimate for the number of zeros of the random function $g_n$ and its derivatives on a period. The approach we adopt is similar to the one recently introduced in \cite{AMP24} and is based on the fact that anti-concentration estimates for sums of independent random vectors improve as dimension grows. Let us first establish the following non-degeneracy lemma. 
\begin{lemma}\label{lem.nondeg}
Suppose that the base function $f$ satisfies Condition $(\mathbf{H2})$, in particular $f$ is of class $\mathcal C^{q}$, with $q=7$. Then, for $n$ large enough, the covariance matrix with respect to $\Prob \otimes \Prob_X$ of the random vector $Z_n=\left(g_n(0),g_n'(0),\dots, g_n^{(q-1)}(0)\right)$ is positive definite. In other words, there exists $\kappa>0$ such that, for $n$ large enough and for all $\lambda=(\lambda_0,\dots, \lambda_{q-1})\in \R^{q}$, we have
\[
\sigma_n^2(\lambda) :=\Esp_{\Prob_X\otimes \Prob}\left[\left(\lambda \cdot Z_n\right)^2\right] \geq \kappa \sum_{i=0}^{q-1} \lambda_i^2 .
\]
\end{lemma}

\begin{proof}
The random vector $Z_n$ is centered under $\Prob_X\otimes \Prob$, and for $\lambda=(\lambda_0,\dots, \lambda_{q-1})\in \R^{q}$, we have 
\begin{eqnarray*}
\sigma_n^2(\lambda) & :=&  \Esp_{\Prob\otimes \Prob_X}\left[ \left(\sum_{i=0}^{q-1}\lambda_i g_n^{(i)}(0) \right)^2\right]= \sum_{i,j=0}^{q-1}\lambda_i \lambda_j \Esp_{\Prob \otimes \Prob_X}\left[g_n^{(i)}(0)g_n^{(j)}(0)\right]\\
&=&\sum_{i,j=0}^{q-1}\lambda_i \lambda_j \Esp_{\Prob \otimes \Prob_X}\left[\frac{1}{n}\left(\sum_{k=1}^{n}\frac{k^{i}}{n^{i}}a_kf^{(i)}(kX)\right) \left(\sum_{r=1}^{n}\frac{r^{j}}{n^{j}}a_rf^{(j)}(rX)\right)\right]\\
&=&\sum_{i,j=0}^{q-1}\lambda_i \lambda_j \frac{1}{n}\sum_{k,r=1}^{n}\frac{k^{i} r^{j}}{n^{i+j}}\Esp[a_ka_r] \Esp_{X}\left[f^{(i)}(kX)f^{(j)}(rX)\right]\\
&=&\sum_{i,j=0}^{q-1}\lambda_i \lambda_j \langle f^{(i)},f^{(j)}\rangle
\frac{1}{n}\sum_{k=1}^{n}\frac{k^{i+j}}{n^{i+j}} =\frac{1}{n}\sum_{k=1}^{n}\left\|\sum_{i=0}^{q-1}\lambda_if^{(i)}\frac{k^{i}}{n^{i}}\right\|^2_{L^2}=\frac{1}{n}\sum_{k=1}^{n}\varphi_q^{\lambda}\left(\frac{k}{n}\right)\end{eqnarray*}
where 
\[
\varphi_q^{\lambda}: x \mapsto \left\|\sum_{i=0}^{q-1}\lambda_i f^{(i)}(\cdot) x^{i}\right\|_{L^2}^{2}
\] 
is a non-negative continuous function on $[0,1]$. Recognizing a Riemann sum, we deduce that the sequence of quadratic forms $(\lambda \mapsto \sigma_n^2(\lambda))_{n\geq 1}$ converge as $n$ goes to infinity to the quadratic form 
\[
\lambda \mapsto \int_0^{1}\varphi_q^{\lambda}(x)dx.
\]
But the latter is positive definite. Indeed, assume by contradiction that it vanishes for $\lambda \in \R^q \backslash \{0\}$, it would imply that $||\varphi_q^{\lambda}(\cdot)||_1=0$ and thus that for Lebesgue almost all $x\in [0,1]$ and for all $\xi \in [0,2\pi]$, 
\begin{equation} \label{eq.sum.deri}
\sum_{i=0}^{q-1}\lambda_i x^{i}f^{(i)}(\xi) = 0.
\end{equation}
Under the assumption \textbf{(H2)}, we could then choose $\xi=\xi_0$ where $f(\xi_0)\neq 0$ in the previous equality (\ref{eq.sum.deri}). Letting $x$ go to zero would then yield
$$\lambda_0 f(\xi_0)=0~~\text{ hence }~\lambda_0=0.$$
In particular, Equality (\ref{eq.sum.deri}) could be rewritten as 
$$\sum_{i=1}^{q-1}\lambda_ix^{i}f^{(i)}(\xi)=0 \Longrightarrow \sum_{i=1}^{q-1}\lambda_i x^{i-1}f^{(i)}(\xi)=0,$$
by dividing both sides by $x \neq 0$. Evaluating the last expression at $\xi=\xi_1$ such that $f'(\xi_1)\neq 0$ and letting again $x$ go to zero would then yield $\lambda_1=0$. Repeating the argument, we would thus obtain $\lambda_0=\lambda_1=\dots=\lambda_{q-1}=0$, hence the contradiction.
\qed 
\end{proof}

\begin{corollary}\label{cor.char}
Under Condition $(\mathbf{H2})$ and with the same notations as in Lemma \ref{lem.nondeg}, there exists $\delta>0$ such that for $n$ large enough and for all $\lambda\in \R^d$ with $||\lambda||\leq \delta$
\[
\left|  \Esp_{\Prob\otimes \Prob_X}\left[ e^{i \lambda \cdot Z_n}\right] \right| \leq e^{-\frac{\kappa }{8}||\lambda||_2^2}.
\]
\end{corollary}

\begin{proof}
Via a second order Taylor expansion of the characteristic function in the neighborhood of zero, we have for $|| \lambda||\leq \delta$ small enough
\[
\left|  \Esp_{\Prob\otimes \Prob_X}\left[ e^{i \lambda \cdot Z_n}\right] \right| \leq 1- \frac{1}{4} \sigma_n^2(\lambda) \leq 1- \frac{\kappa}{4} ||\lambda||^2 \leq  e^{-\frac{\kappa }{8}||\lambda||_2^2}.
\]
\qed
\end{proof}

\begin{proposition}\label{pro.kolmorogo}
Under Condition $(\mathbf{H2})$ and with the same notations as in Lemma \ref{lem.nondeg}, there exists a constant $C_{q}$ independent of $n$ such that for $n$ large enough 
\[
\mathbb P \otimes \mathbb P_X\left(  ||Z_n||\leq \frac{1}{\sqrt{n}} \right) \leq \frac{C_q}{n^{q/2}}.
\]
\end{proposition}

\begin{proof}
We start by recalling Esseen concentration inequality, see e.g. Lemma 7.17 of \cite{MR2289012}, if $X$ is a random vector in $\R^q$ defined on a probability space $(E, \mathcal E, \mathbb P)$, then for any ${r>0}$ and $ {\varepsilon > 0}$, 
\[
\displaystyle \sup_{x \in {\R}^q} \mathbb P ( ||X - x|| \leq r ) \leq \kappa_{q,\varepsilon} r^q \int_{\substack{s \in {\R}^q\\ ||s|| \leq \varepsilon/r}}  \left|\mathbb E\left[  e^{i s \cdot X}\right]\right| ds,
\]
for some constant ${\kappa_{d,\varepsilon}}$ depending only on ${d}$ and ${\varepsilon}$.
In our context, we have thus in particular via Corollary \ref{cor.char}, for $n$ large enough, $r=1$ and $\varepsilon=\delta$ 
\[
\begin{array}{ll}
\displaystyle \mathbb P \otimes \mathbb P_X\left( || Z_n|| \leq \frac{1}{\sqrt{n}} \right)&  \displaystyle \leq \kappa_{q,\varepsilon}  \int_{\substack{s\in {\R}^q\\ ||s|| \leq \delta}}  \left|\Esp_{\Prob\otimes \Prob_X} \left[  e^{i s \cdot \sqrt{n} Z_n }\right]\right| ds \leq \kappa_{q,\delta}  \int_{\substack{s \in {\R}^q\\ ||s|| \leq \delta}}  e^{-\frac{\kappa }{8}n ||s||_2^2} ds\\
\\
& \leq\displaystyle \frac{\kappa_{q,\delta} }{n^{q/2}} \int_{s \in {\R}^q}  e^{-\frac{\kappa }{8}||s||_2^2} ds=: \frac{C_q}{n^{q/2}}.
\end{array}
\]
\qed
\end{proof}

\subsection{Application to average nodal asymptotics} \label{sec.ui}
Let us first give an \emph{a priori} bound on the number of zeros. Assume \textbf{(H2)}, in particular the general function $f$ can be written as a piecewise polynomial function with $ \mathcal{C}^{q}([0,2\pi])$ regularity, with $q=7$. More precisely, there exist an integer $M$, polynomial functions $P_1,\dots, P_M$ and knots $s_1=s_{M+1}<s_2<\dots<s_M$ in $[0,2\pi]$, with $f=\sum_{i=1}^M P_i \mathds{1}_{[s_i, s_{i+1}[}$ (with periodicity).
\begin{lemma}\label{lm.apriori} $\Prob-$almost surely, we have
$$\mathcal{N}(f_n,[0,2\pi]) = O(n^2).$$
\end{lemma}
\begin{proof}
The set $\mathcal{S}:=\{\sigma_1,\sigma_2, \dots,\}$ of the  knots of the function $f_n(t)=\frac{1}{\sqrt{n}} \sum_{k=1}^{n}a_kf(kt)$ are the elements $t \in [0,2\pi]$ such that there exists $k \in \{1,\dots, n\}$ for which $kt=s_i\mod 2 \pi$, for some $i \in \{1,\dots, M\}$. Since for $1\leq k \leq n$, any possible knot $t=\frac{s_i}{k}$ gives itself $k$ knots by $2\pi$-periodicity, we deduce that $\text{Card}(\mathcal{S}) \leq Mn^2$. Now, for $i=1,\dots, \text{Card}(\mathcal{S})-1$, the restriction of $f_n$ on the interval $[\sigma_i, \sigma_{i+1}]$, is a polynomial of degree at most $d:=\max(d_1,\dots,d_M)$ hence has at most $d$ zeros on $[\sigma_i,\sigma_{i+1}]$. Therefore, 
$$\mathcal{N}(f_n,[0,2\pi]) \leq Mn^2\times d=O(n^2).$$
\qed
\end{proof}
Remark that, contrary to the trigonometric case, the sequence $n^{-1}\mathcal{N}(f_n,[0,2\pi])$ is here not bounded, so that establishing its uniform integrability requires a careful examination.

\begin{proposition}\label{prop.eq.int}
Under \textbf{(H2)}, there exists a constant $\eta>0$ such that
$$\sup_{n\ge 1}\Esp_{\Prob \otimes \Prob_X}\left[\left|\mathcal{N}(g_n,[0,2\pi])\right|^{1+\eta}\right]<+\infty.$$
\end{proposition}

\begin{proof}
Using the previous quadratic bound for the number of real zeros, Fubini inversion gives
\[
\Esp_{\Prob \otimes X} \left[\left|\mathcal{N}(g_n,[0,2\pi]) \right|^{1+\eta}\right]=(1+\eta)\sum_{p=0}^{O(n^{2})} p^{\eta} \Prob\otimes \Prob_X\left(\mathcal{N}(g_n,[0,2\pi])>p\right).
\]
Since we interested in the finiteness of the last sum, we can of course skip a finite number of terms. Let us fix $\alpha<2$ to be determined latter and first consider the case where $25 \leq p \leq n^{\alpha}$. By the pigeon hole principle, dividing the period $[0,2\pi]$ in $\lfloor \sqrt{p}\rfloor $ intervals of size $\delta_p:=1/\lfloor \sqrt{p}\rfloor$, on the event $\{ \mathcal{N}(g_n,[0,2\pi])>p\}$, there exists a small interval which contains at least $\lfloor \sqrt{p}\rfloor$ zeros. Under $\Prob\otimes\Prob_X$, we have stationarity of the number of zeros of $g_n$ in intervals of given length so that with a union bound we get
\begin{eqnarray*}
\Prob\otimes\Prob_X (\mathcal{N}(g_n,[0,2\pi])>p) & \leq & \Prob \otimes\Prob_X\left( \bigcup_{k=1}^{\lfloor \sqrt{p}\rfloor}\mathcal{N}\left(g_n,\left[2\pi (k-1)\delta_p, 2\pi k \delta_p\right ]\right)> \lfloor \sqrt{p}\rfloor \right) \\
&\leq &\sum_{k=1}^{\left\lfloor \sqrt{p} \right\rfloor}     \Prob \otimes\Prob_X\left(    \mathcal{N}\left(g_n,\left[2\pi (k-1)\delta_p, 2\pi k \delta_p \right ]\right)> \lfloor \sqrt{p}\rfloor \right) \\
\\
& \leq & \lfloor \sqrt{p}\rfloor \times \Prob \otimes\Prob_X\left(    \mathcal{N}\left(g_n,\left[0,  2\pi \delta_p \right ]\right)> \lfloor \sqrt{p}\rfloor \right).
\end{eqnarray*} 
We can upper bound the last probability using the now standard iterated Rolle Lemma procedure. Namely, we have 
$$\left\{ \mathcal{N}(g_n,[0,2\pi\delta_p])>5\right\} \subseteq \bigcap_{j=0}^{4}\left\{ \mathcal{N}(g_n^{(j)},[0,2\pi\delta_p])>5-j\right\},$$
and for all $j \in \{0, \dots, 4\}$, on this event, by Taylor--Lagrange approximation, we have
$$\sup_{x \in [0,2\pi\delta_p]}|g_n^{(j)}(0)| \leq \frac{(2\pi \delta_p)^{5-j}}{(5-j)!}\|g_n^{(5)}\|_{\infty}~~~,~~\text{with}~\|\cdot\|_{\infty}=\|\cdot\|_{L^{\infty}([0,2\pi])}.$$
As a result for $p \geq 25$, 
\begin{eqnarray*}
& & \Prob \otimes \Prob_X\left(\mathcal{N}(g_n,[0,2\pi\delta_p])>\lfloor \sqrt{p} \rfloor\right) \leq \Prob \otimes \Prob_X\left(  \mathcal{N}(g_n,[0,2\pi\delta_p])>5\right)  \\
&\leq& \Prob\otimes \Prob_X\left(|g_n(0)| \leq \frac{(2\pi \delta_p)^{5}}{5!}\|g_n^{(5)}\|_{\infty}\right) \\
&\leq& \Prob \otimes \Prob_X\left(|g_n(0)|\leq \frac{(2\pi \delta_p)^{5}}{5!}M\right) + \Prob \otimes \Prob_X\left(\|g_{n}^{(5)}\|_{\infty} >M\right),
 \end{eqnarray*}
where $M\geq 0$ is arbitrary.  
Let us first estimate the probability $\displaystyle \Prob \otimes \Prob_X\left(\|g_n^{(5)}\|_{\infty}>M\right)$,
for which Markov inequality yields
\[
\Prob \otimes \Prob_X\left(\|g_n^{(5)}\|_{\infty}>M\right) \leq \frac{\Esp_{\Prob \otimes \Prob_X}\left[ \|g_n^{(5)}\|_{\infty}^2\right]}{M^2}.
\]
Comparing the uniform norm with Sobolev norm, see e.g. Lemma 5.15 p.107 of \cite{Ada03}, there exists a universal constant $C>0$ such that
\[
\Esp_{\Prob \otimes \Prob_X}\left[\|g_n^{(5)}\|_{\infty}^2\right]\leq C \left( \Esp_{\Prob \otimes \Prob_X}\left[\|g_n^{(5)}\|^2_{2}\right]+\Esp_{\Prob\otimes \Prob_X}\left[\|g_n^{(6)}\|^2_2\right] \right).
\]
Using Fubini inversion, for $\ell \in \{5,6\}$, we get
\begin{eqnarray*}
\Esp_{\Prob\otimes \Prob_X}\left[\|g_n^{(\ell)}\|^2_2\right]&=&\Esp_{\Prob\otimes \Prob_X}\left[\frac{1}{2\pi}\int_0^{2\pi}|g_n^{(\ell)}(t)|^2dt\right]\\
&=&\frac{1}{2\pi}\int_0^{2\pi}\Esp_{\Prob\otimes \Prob_X}\left[|g_n^{(\ell)}(t)|^2\right]dt.
\end{eqnarray*}
We now state the following result, extending the uniform bound of Lemma \ref{lem.controlL2} to the derivatives of $\{g_n(t)\}_{t\ge 0}$. Its proof is postponed in Section \ref{sec.tech}.
\begin{lemma} \label{lm.est.bir.gen}
There exists $C_f>0$ such that for $\ell \in \{5,6\}$
$$\sup_{n\ge 1}\sup_{t\in [0,2\pi]}\Esp_{\Prob\otimes \Prob_X}\left[|g_n^{(\ell)}(t)|^2\right]\leq C_f.$$
\end{lemma}
\noindent
We thus have the existence of a constant $C(f)>0$ only depending on $f$ such that
\begin{equation}\label{eq.probgn}\Prob\otimes \Prob_X(\|g_n^{(5)}\|_{\infty}>M)\leq \frac{C(f)}{M^2}.\end{equation}
Otherwise, recalling that $g_n(0)=f_n(X)$ and using standard Berry--Esseen bounds which can be established in the same manner as the CLT established in Theorem \ref{thm.salem.func.gen}, we have if $Z \sim \mathcal N(0,1)$
\[
\left| \Prob \otimes \Prob_X\left(|g_n(0)|\leq \frac{(2\pi \delta_p)^{5}}{5!}M\right) - \Prob \otimes \Prob_X\left(|Z|\leq \frac{(2\pi \delta_p)^{5}}{5!}M\right) \right| = O\left(  \frac{1}{\sqrt{n}} \right) .
\]
As a result, we get 
\begin{equation}\label{esti2}
 \Prob \otimes \Prob_X\left(|g_n(0)|\leq \frac{(2\pi \delta_p)^{5}}{5!}M\right) = O\left(  \delta_p^5 M \right)  + O\left(  \frac{1}{\sqrt{n}} \right) .
\end{equation}
Combining the estimates \eqref{eq.probgn} and \eqref{esti2}, and optimizing in $M$, we deduce
\[
\begin{array}{ll}
\displaystyle{\Prob \otimes \Prob_X\left(\mathcal{N}(g_n,[0,2\pi\delta_p])>\lfloor \sqrt{p} \rfloor\right)} & \displaystyle{= O\left(  \delta_p^5 M \right)  + O\left(  \frac{1}{\sqrt{n}} \right)+ O\left( \frac{1}{M^2}\right)} \\
\\
& \displaystyle{= O\left(  \delta_p^{\frac{10}{3}} \right)  + O\left(  \frac{1}{\sqrt{n}} \right)}.
\end{array}
\]
Recalling that $\delta_p=1/\lfloor \sqrt{p} \rfloor$, we get for $25 \leq p \leq n^{\alpha}$
\[
\begin{array}{ll}
p^{\eta} \displaystyle{\Prob\otimes\Prob_X (\mathcal{N}(g_n,[0,2\pi])>p)} & = p^{\eta} \displaystyle{\lfloor \sqrt{p} \rfloor \left(   O\left(  \delta_p^{\frac{10}{3}} \right)  + O\left(  \frac{1}{\sqrt{n}} \right) \right) }\\
\\
& \displaystyle{=   O\left(  \frac{1}{p^{\frac{7}{6}-\eta}} \right)  + O\left(  p^{\eta} \frac{\lfloor \sqrt{p} \rfloor}{\sqrt{n}} \right).  }
\end{array}
\]
Therefore, for $0<\eta<1/6$ and $\alpha < 1/(3+2\eta)$ so that $\alpha ( \eta+3/2)-1/2<0$, we have as $n$ goes to infinity
\[
\sum_{p=25}^{n^{\alpha}} p^{\eta} \displaystyle{\Prob\otimes\Prob_X (\mathcal{N}(g_n,[0,2\pi])>p)} = O(1) + O\left(  n^{\alpha ( \eta+3/2)-1/2}\right) = O(1)+o(1).
\]
Let us now consider the case where $p \geq n^{\alpha}$. We fix $0<\beta<\alpha$ and proceed as above, namely we divide the whole period $[0,2\pi]$ in small intervals of size $1/\lfloor n^{\beta}\rfloor$. On the event $\{\mathcal{N}(g_n,[0,2\pi])>p\}$, there exists a small interval on which $g_n$ has at least $n^{\alpha-\beta}$ zeros. Using again an union bound and the stationarity of the fields under $\Prob\otimes\Prob_X$, we get this time
\begin{eqnarray*}
\Prob\otimes\Prob_X (\mathcal{N}(g_n,[0,2\pi])>p) & \leq &  \lfloor n^{\beta} \rfloor \times \Prob \otimes\Prob_X\left(    \mathcal{N}\left(g_n,\left[0,  \frac{2\pi}{\lfloor n^{\beta}  \rfloor} \right ]\right)> \lfloor n^{\alpha-\beta} \rfloor \right).
\end{eqnarray*} 
Since $\beta<\alpha$, we have $n^{\alpha-\beta}>2q$ for $n$ large enough for any given integer $q$ to be fixed later, and iterating $q-$times Rolle Lemma, we deduce as above that, for $M>0$ to be specified later, this last expression can be upper bounded by 
\begin{eqnarray*}
 \lfloor n^{\beta} \rfloor \left[ \Prob\otimes\Prob_X \left(  ||(g_n(0), \ldots, g_n^{(q-1)}(0)) ||_{\infty} \leq \left(\frac{2\pi}{\lfloor n^{\beta} \rfloor}\right)^q M\right)  + O\left(\frac{1}{M^2}\right) \right].
\end{eqnarray*} 
Note that, as above, the last bound $O(1/M^2)$ is obtained comparing uniform and  Sobolev norms, and thus requires the base signal $f$ to be of class not only $\mathcal C^q$ but $\mathcal C^{q+1}$.
Let us now choose $M$ of the form $M=n^{\gamma}$ and $q$ large enough that $\gamma-\beta q \leq -1/2$ so that 
\[
\left(\frac{2\pi}{\lfloor n^{\beta} \rfloor}\right)^q M = O\left(  \frac{1}{\sqrt{n}}\right).
\]
Then, using the small ball estimate established in Proposition \ref{pro.kolmorogo}, we get 
\[
\Prob\otimes\Prob_X \left(  ||(g_n(0), \ldots, g_n^{(q-1)}(0)) ||_{\infty} \leq \left(\frac{2\pi}{\lfloor n^{\beta} \rfloor}\right)^q M\right) = O\left( \frac{1}{n^{q/2}}\right). 
\]
As a result, we get 
\[
\sum_{p=n^{\alpha}}^{O(n^2)} p^{\eta} \Prob\otimes\Prob_X (\mathcal{N}(g_n,[0,2\pi])>p) = O\left(  n^{2(1+\eta)+\beta - q/2} \right)  + O\left(  n^{2(1+\eta)+\beta - 2\gamma} \right).
\]
We are then left to optimize the parameters $\eta, \alpha, \beta, \gamma, q$. We can take $\eta$ arbitrary small but positive, and thus $\alpha$ arbitrary close to $1/3$. Then $\beta$ can be chosen arbitrary close to $\alpha$, and the constraints $2(1+\eta)+\beta - q/2\leq 0$ and $2(1+\eta)+\beta - 2\gamma \leq 0$ yield $q>14/3$, $\gamma>14/12$. In fine, the constraint $\gamma-\beta q \leq -1/2$ yield $q>5$, so that $q\geq 6$ is suitable. Remembering that the comparison of uniform and Sobolev norms requires a regularity exponent $q+1$, the above proof is licit as soon as the base function has regularity $\mathcal C^{7}$.
\qed
\end{proof}

Combining the last uniform integrability estimate established in Proposition \ref{prop.eq.int} with the convergence in distribution of the number of real zeros established in the second point of Corollary \ref{cor.conv.nb.zeros}, we obtain the convergence of the first moment.
\begin{corollary}\label{cor.conv.mom} Under \textbf{(H1)} and  \textbf{(H2)}, assuming that the common law of the $a_k$ admits a third moment,
$$\lim_{n\to +\infty}\Esp_{\Prob \otimes \Prob_X}\mathcal{N}(g_n,[0,2\pi]) =\Esp_{\Prob \otimes \Prob_X} \mathcal{N}(g_{\infty},[0,2\pi]).$$
\end{corollary}
\noindent
Combining Equation (\ref{eq.nb.zero.g}) and Corollary \ref{cor.conv.mom}, we thus get that 
$$\lim_{n \to +\infty} \frac{\Esp \left[\mathcal{N}(f_n,[0,2\pi])\right]}{n} = \Esp_{\Prob \otimes \Prob_X}\left[\mathcal{N}(g_{\infty},[0,2\pi])\right].$$
Thanks to Lemma \ref{lm.nd.kac}, we have
$$\Esp_{\Prob \otimes \Prob_X}\mathcal{N}(g_{\infty},[0,2\pi])=\frac{2}{\sqrt{3}}\sqrt{\frac{\langle f',f' \rangle}{\langle f,f\rangle}},$$
which conclude the proof of Theorem \ref{thm.mean.num}.
More generally, a similar proof would give that for any interval $[a,b]\subset [0,2\pi]$, 
$$\lim_{n \to+\infty}\frac{\Esp[\mathcal{N}(f_n,[a,b])]}{n} =\frac{b-a}{\pi \sqrt{3}} \sqrt{\frac{\langle f',f'\rangle}{\langle f,f\rangle}}.$$

\section{Proofs of technical lemmas} \label{sec.tech}
This last section is devoted to the proofs of the different technical lemmas stated in Section 2 and dealing with the convergence of the variable $f_n(X)$ to a Gaussian limit and more generally of the convergence of the process $g_n(t)=f_n(X+t/n)$ to the universal Gaussian process $g_{\infty}$. 

\subsection{Proof of Lemma \ref{lem.delta}} \label{sec.appdelta}
We give here the detailed proof of Lemma \ref{lem.delta}. Namely, recalling that 
\[
\Delta_n({\lambda}):=\Esp\left[ \left|\Esp_X\left[e^{i {\lambda} f_n(X)}-e^{-\frac{{\lambda}^2}{2}\sigma_n^2(X)]}\right]\right|^2\right],
\]
we want to establish that if $a_1$ admits a finite third moment and if $f \in H^1$, there exists a positive constant $C$ such that 
\[
\Delta_n({\lambda}) \leq \frac{C ({\lambda}^2+|{\lambda}|^3)}{\sqrt{n}}.
\]
Developing the square under the expectation, if $Y$ an independent copy of $X$, and if $\mathbb E_{X,Y}$ denotes the expectation with respect to $\mathbb P_X \otimes \mathbb P_Y$, by inverting the sums, we obtain
\begin{eqnarray*}
\Delta_n({\lambda})&=&\Esp_{X,Y}\left[\Delta_{n,1}({\lambda})\right]+\Esp_{X,Y}\left[\Delta_{n,2}({\lambda})\right]+\Esp_{X,Y}\left[\Delta_{n,3}({\lambda})\right],
\end{eqnarray*}
where
\[\begin{array}{ll}
\Delta_{n,1}({\lambda}):=& \Esp[e^{i{\lambda}(f_n(X)-f_n(Y))}]-e^{-\frac{{\lambda}^2}{2}\Esp[(f_n(X)-f_n(Y))^2]},\\
\\
\Delta_{n,2}({\lambda}):=& e^{-\frac{{\lambda}^2}{2}\Esp[(f_n(X)-f_n(Y))^2]} -   e^{-\frac{{\lambda}^2}{2}(\sigma_n^2(X)+\sigma_n(Y)^2)},\\
\\
\Delta_{n,3}({\lambda}):=&-\left(\Esp\left[e^{i {\lambda} f_n(X)}\right]-e^{-\frac{{\lambda}^2}{2}\sigma_n^2(X)}\right)e^{-\frac{{\lambda}^2}{2}\sigma_n^2(Y)} \\
\\ & - \left(\Esp\left[e^{-i {\lambda} f_n(Y)}\right]-e^{-\frac{{\lambda}^2}{2}\sigma_n^2(Y)}\right)e^{-\frac{{\lambda}^2}{2}\sigma_n^2(X)}.
\end{array}\]
Let us first consider the term $\Delta_{n,1}({\lambda})$. 
On the one hand, if we denote by $\phi$ the characteristic function of $a_1$, we have by independence of the coefficients $a_k$ that
\[
\Esp[e^{i{\lambda} (f_n(X)-f_n(Y))}]=\prod_{k=1}^{n}\phi\left(\frac{{\lambda}}{\sqrt{n}} (f(kX)-f(kY))\right).
\]
On the other hand, we have 
\[
\Esp[(f_n(X)-f_n(Y))^2] = \frac{1}{n}\sum_{k=1}^{n}(f(kX)-f(kY))^2.
\]
Since $a_1$ has a finite third moment, the function $\log(\phi)$ is three-time differentiable in the neighborhood of zero. Using Taylor-Lagrange inequality and the fact that the $a_k$ are centered with variance one, there exists constants $\eta>0$ and $C>0$ such that for 
\[
\left|\log \phi(u) +\frac{u^2}{2}\right| < C|u|^3, \quad \forall |u| \leq \eta.
\]
By Sobolev embedding, if $f \in H^1$, it is continuous and hence bounded on $[0, 2\pi]$. Therefore, for all ${\lambda} \in \mathbb R$ and for $n$ large enough we have, uniformly in $1\leq k  \leq n$, 
\[
\left| \frac{{\lambda}}{\sqrt{n}}(f(kX)-f(kY)) \right| \leq \eta
\]
and thus uniformly in $1\leq k  \leq n$
\[
\left| \log\left(\phi\left(\frac{{\lambda}}{\sqrt{n}}(f(kX)-f(kY))\right) \right) +\frac{1}{2}\frac{{\lambda}^2}{n}(f(kX)-f(kY))^2\right| \leq \frac{8C||f||_{\infty}^3 |{\lambda}|^3}{n\sqrt{n}}.
\]
Therefore, the exponential being Lipschitz, we get 
\[
\left|\prod_{k=1}^{n}\phi\left(\frac{{\lambda}}{\sqrt{n}}(f(kX)-f(kY))\right)-e^{-\frac{{\lambda}^2}{2}\Esp[(f_n(X)-f_n(Y))^2]}\right| \leq \frac{8C||f||_{\infty}^3 |{\lambda}|^3}{\sqrt{n}},
\]
and taking the expectation with respect to $\mathbb P_X \otimes \mathbb P_Y$, we deduce that 
\[
|\Esp_{X,Y}\left[\Delta_{n,1}({\lambda})\right]| \leq \frac{8C||f||_{\infty}^3 |{\lambda}|^3}{\sqrt{n}}.
\]
Proceeding in the exact same way, we have
\[
|\Esp_{X,Y}\left[\Delta_{n,3}({\lambda})\right]| \leq \frac{2C||f||_{\infty}^3 |{\lambda}|^3}{\sqrt{n}}.
\]
We are thus left with the term $\Delta_{n,2}({\lambda})$. Using again the fact that the exponential function is Lipschitz, we have 
\[
|\Esp_{X,Y}\left[\Delta_{n,2}({\lambda})\right] |\leq {\lambda}^2\Esp_{X,Y} \left|\Esp[f_n(X)f_n(Y)]\right|.
\]
The conclusion of Lemma \ref{lem.delta} then follows from the following decorrelation estimate.
\begin{lemma}{Suppose that the base function $f$ satisfies the condition $\mathbf{(H1)}$}, then there exists a constant $C=C(||f||_{H^1})$ such that
\[
\mathbb E_{X,Y} \left[ \left|\mathbb E[f_n(X)f_n(Y)] \right|\right]\leq \frac{C}{\sqrt{n}}.
\]\label{lem.decorel}
\end{lemma}
\begin{proof}[Proof of Lemma \ref{lem.decorel}]
We have
\[
\mathbb E[f_n(X)f_n(Y)]=\frac{1}{n}\sum_{k,\ell=1}^{n}\mathbb E[a_ka_{\ell}]f(kX)f(\ell Y)=\frac{1}{n}\sum_{k=1}^{n}f(kX)f(kY).
\]
{Since the base function $f$ satisfies the condition $\mathbf{(H1)}$, it is the limit of its Fourier sums $S_N f(x):=\sum_{|p| \leq N} \widehat{f}(p) e^{i p x}$, and since it is centered, we have $ \widehat{f}(0) =0$}. For any $N$, we can then write 
\[
\begin{array}{ll}
\mathbb E[f_n(X)f_n(Y)]  =\displaystyle{\frac{1}{n}\sum_{k=1}^{n} \left( f(kX)-S_N f(kX)\right) f(kY)}\\
\\
  + \displaystyle{\frac{1}{n}\sum_{k=1}^{n} S_N f(kX) \left(f(kY)-S_N f(kY)\right)
  +\frac{1}{n}\sum_{k=1}^{n} S_N f(kX) S_N f(kY).}
\end{array}
\]
Now taking the absolute value and the expectation under $\mathbb E_{X,Y}$, we get 
\[
\begin{array}{l}
\mathbb E_{X,Y} \left[ \left|\mathbb E[f_n(X)f_n(Y)] \right|\right] \leq \displaystyle{\frac{1}{n}\sum_{k=1}^{n} \mathbb E_X\left[ \left| f(kX)-S_N f(kX)\right|\right] \mathbb E_Y \left[\left|f(kY)\right|\right]}\\
\\
  + \displaystyle{\frac{1}{n}\sum_{k=1}^{n} \mathbb E_X\left[|S_N f(kX)| \right] \mathbb E_Y\left[ |f(kY)-S_N f(kY)|\right]} 
+ \displaystyle{\mathbb E_{X,Y}\left[ \left| \frac{1}{n}\sum_{k=1}^{n} S_N f(kX) S_N f(kY)\right|\right].}
\end{array}
\]
Applying Cauchy--Schwarz inequality, we get for the first term on the right hand side
\[
\frac{1}{n}\sum_{k=1}^{n} \mathbb E_X\left[ \left| f(kX)-S_N f(kX)\right|\right] \mathbb E_Y \left[\left|f(kY)\right|\right] \leq   ||f-S_Nf ||_2 \times ||f||_2.
\]
The second term can be treated similarly, namely 
\[
\frac{1}{n}\sum_{k=1}^{n} \mathbb E_X\left[|S_N f(kX)| \right] \mathbb E_Y\left[ |f(kY)-S_N f(kY)|\right]\leq \, ||f||_2 \times  ||f-S_Nf ||_2.
\]
For the third term, {recalling that $ \widehat{f}(0)=0$}, observe that 
\[
S_Nf(kX) S_Nf(k Y)
= \sum_{1\leq |p|\leq N,1\leq |q|\leq N}\widehat{f}(p)\widehat{f}(-q)e^{ik(pX-qY)} 
\]
and
$$\frac{1}{n}\sum_{k=1}^{n}e^{ik(pX-qY)}=\exp\left(i\frac{n+1}{2}(pX-qY)\right)\frac{1}{n}\frac{\sin\left(\frac{n}{2}(pX-qY)\right)}{\sin\left(\frac{pX-qY}{2}\right)}.$$
Therefore, applying once again Cauchy--Schwarz inequality (twice indeed), we deduce that if $K_n$ denotes the standard Fej\'er such that $||K_n||_1=1$
\[
\begin{array}{l}
\displaystyle{\frac{1}{n}\sum_{k=1}^{n} \mathbb E_{X,Y}\left[|S_N f(kX) S_N f(kY)|\right]}  \displaystyle{\leq \sum_{\substack{1\leq |p|\leq N, \\ 1\leq |q|\leq N}}|\widehat{f}(p)\widehat{f}(-q)|\mathbb E_{X,Y}\left[\frac{1}{n}K_n(pX-qY)\right]^{1/2}}\\
\\
 =\displaystyle{ \frac{1}{\sqrt{n}} \left( \sum_{1\leq |p|\leq N}|\widehat{f}(p)|\right)^2 \leq \frac{1}{\sqrt{n}} \left( \sum_{1\leq |p|\leq N}|\widehat{f'}(p)|^2\right)\left( \sum_{1\leq |p|\leq N}\frac{1}{p^2}\right)\leq \frac{2 ||f'||_2^2}{\sqrt{n}}}.
\end{array}
\]
We thus get 
\[
\mathbb E_{X,Y} \left[ \left|\mathbb E[f_n(X)f_n(Y)] \right|\right] \leq 2\, ||f||_2 \times ||f-S_Nf ||_2 + \frac{2 ||f'||_2^2}{\sqrt{n}}.
\]
Now by Parseval identity, we have 
\[
||f-S_Nf ||_2^2 = \sum_{|p|>N}|\widehat{f}(p)|^2 =  \sum_{|p|>N}|\widehat{f'}(p)|^2 \times \frac{1}{p^2} \leq \frac{||f'||_2^2}{N^2}.
\]
As a conclusion, choosing the threshold $N=\lfloor\sqrt{n}\rfloor$, we get indeed that there exists a constant $C=C(||f||_{H^1})$ such that 
\[
\mathbb E_{X,Y} \left[ \left|\mathbb E[f_n(X)f_n(Y)] \right|\right]\leq \frac{C}{\sqrt{n}}.
\]
\end{proof}

\subsection{Proof of Lemma \ref{lem.controlL2}}\label{sec.appcontrolL2}

We now give {the proof} of Lemma \ref{lem.controlL2}. Recall that the goal is to establish that there exists a constant $C=C(\omega)$ such that $\mathbb P-$almost surely 
\[
\sup_{n \geq 1} \mathbb E_X [f_n(X)^2] \leq C, \quad \sup_{\substack{n \geq 1,\\ M>n}} \mathbb E_X \left[\left(\frac{1}{\sqrt{n}} \sum_{M\leq k \leq M+n}a_k f_k(X)\right)^2\right] \leq C.
\]
We will give the proof of the first estimate as the proof of the second follows the exact same lines. 
Developing the square, we have 
\[
\mathbb E_X [f_n(X)^2]=  \frac{1}{n} \sum_{k, \ell=1}^n a_k a_{\ell} \mathbb E_X[ f(kX)f(\ell X)]. 
\]
As in the proof of Lemma \ref{lem.decorel} above, we can approximate $f$ by its Fourier sum $S_N f$ and decompose the product below as
\[
\begin{array}{ll}
f(kX)f(\ell X) & = \left(f(kX)-S_N f(kX)\right) f(\ell X) + S_N f(kX) \left(f(\ell X)-S_N f(\ell X)  \right) \\
\\
& +S_N f(kX) S_N f(\ell X),
\end{array}
\]
so that 
\[
\begin{array}{l}
\displaystyle{\mathbb E_X [f_n(X)^2]}  \displaystyle{= \frac{1}{n} \sum_{k, \ell=1}^n a_k a_{\ell} \mathbb E_X \left[ \left( f(kX)-S_N f(kX)\right) f(\ell X)\right] }\\
\\
+\displaystyle{ \frac{1}{n} \sum_{k, \ell=1}^n a_k a_{\ell} \mathbb E_X \left[ S_N f(kX) \left( f(\ell X)-S_N f(\ell X)  \right) \right]}+\displaystyle{ \frac{1}{n} \sum_{k, \ell=1}^n a_k a_{\ell} \mathbb E_X \left[ S_N f(kX) S_N f(\ell X)\right]}.
\end{array}
\]
By Cauchy--Schwarz inequality, we have
\[
\left|  \frac{1}{n} \sum_{k, \ell=1}^n a_k a_{\ell} \mathbb E_X \left[ \left( f(kX)-S_N f(kX)\right) f(\ell X)\right]\right| \leq \left(  \frac{1}{n} \sum_{k, \ell=1}^n |a_k a_{\ell} |\right) ||f-S_N  f||_2 \times ||f||_2, 
\]
and similarly 
\[
\left| \frac{1}{n} \sum_{k, \ell=1}^n a_k a_{\ell} \mathbb E_X \left[ S_N f(kX) \left( f(\ell X)-S_N f(\ell X)  \right) \right]\right| \leq \left(  \frac{1}{n} \sum_{k, \ell=1}^n |a_k a_{\ell} |\right) ||f-S_N  f||_2 \times ||f||_2.
\]
For the last term, we use Minkowski inequality to deduce
\[
\begin{array}{l}
\displaystyle{\frac{1}{n} \sum_{k, \ell=1}^n a_k a_{\ell} \mathbb E_X \left[ S_N f(kX) S_N f(\ell X)\right] } = \displaystyle{\mathbb E_X \left[ \left(\frac{1}{\sqrt{n}} \sum_{k=1}^n a_k S_N f (kX)\right)^2\right]}\\
\\
=  \displaystyle{\mathbb E_X \left[ \left|\sum_{|p|\leq N} \widehat{f}(p) \left( \frac{1}{\sqrt{n}} \sum_{k=1}^n a_k e^{i kp X}\right)\right|^2\right]} \leq \displaystyle{\left[ \sum_{|p|\leq N} |\widehat{f}(p)|  \mathbb E_X \left[ \left|  \frac{1}{\sqrt{n}} \sum_{k=1}^n a_k e^{i kp X}\right|^2\right]^{1/2} \right]^2} \\
\\
 = \displaystyle{\left( \frac{1}{n} \sum_{k=1}^n a_k^2\right) \left( \sum_{|p|\leq N} |\widehat{f}(p)| \right)^2 \leq \left( \frac{1}{n} \sum_{k=1}^n a_k^2\right) \times 2 || f'||_2^2.}
\end{array}
\]
Therefore, we have 
\[
\mathbb E_X [f_n(X)^2] \leq \left(  \frac{1}{n} \sum_{k, \ell=1}^n |a_k a_{\ell} |\right) ||f-S_N  f||_2 \times ||f||_2 + \left( \frac{1}{n} \sum_{k=1}^n a_k^2\right) \times 2 || f'||_2^2.
\]
Choosing $N=n$, so that $||f-S_N  f||_2  =O(1/n)$, by the law of large numbers, we conclude that there exists a constant $C=C(\omega)$ such that $\mathbb P-$almost surely 
\[
\sup_{n \geq 1} \mathbb E_X [f_n(X)^2] \leq C.
\]
\subsection{Proof of Lemma \ref{lem.first reduc}}\label{sec.first reduc}
This section is devoted to the proof of the continuity Lemma \ref{lem.first reduc}, namely our goal is to establish that $\mathbb P-$almost surely, there exists a constant $C=C(\omega)$ such that 
\[
\left| \Esp_X \left[e^{iZ_n(X,t,\lambda)} \right] -\Esp_X \left[e^{iZ_n(X,s,\mu)} \right]  \right|\leq C\left( ||\lambda-\mu||_1 +||\lambda||_1 \times ||t-s||_{\infty}\right),
\]
where we recall that 
\[
Z_n(X,t,\lambda)=\sum_{j=1}^{M}\lambda_j f_n\left(X+\frac{t_j}{n}\right), \quad Z_n(X,s,\mu)=\sum_{j=1}^{M}\mu_j f_n\left(X+\frac{s_j}{n}\right).
\]
We have first, the exponential being Lipschitz and by the triangular inequality
\[
\begin{array}{ll}
\displaystyle{\left| \Esp_X \left[e^{iZ_n(X,t,\lambda)} \right] -\Esp_X \left[e^{iZ_n(X,s,\mu)} \right]  \right| \leq\Esp_X \left[\left| Z_n(X,t,\lambda)-  Z_n(X,s,\mu) \right|\right]}\\
\\
\displaystyle{\leq \sum_{j=1}^M |\lambda_j| \, \mathbb E_X \left[\left| f_n\left(X+\frac{t_j}{n}\right)-f_n\left(X+\frac{s_j}{n}\right)\right|  \right]+\sum_{j=1}^M |\lambda_j-\mu_j| \, \mathbb E_X \left[\left| f_n\left(X+\frac{t_j}{n}\right)\right|\right]}.
\end{array}
\]
On the one hand, by Cauchy--Schwarz inequality and Lemma \ref{lem.controlL2}, $\mathbb P-$almost surely, there exists a constant $C$ such that 
\[
\sup_{ n\geq 1} \mathbb E_X \left[\left| f_n\left(X+\frac{t_j}{n}\right)\right|\right] \leq \sup_{ n\geq 1} \mathbb E_X \left[\left| f_n\left(X\right)\right|^2\right]^{1/2}\leq C.
\]
On the other hand, by Cauchy--Schwarz inequality again, we have for $1 \leq j \leq M$ 
\[
\begin{array}{ll}
\displaystyle{ \mathbb E_X \left[\left| f_n\left(X+\frac{t_j}{n}\right)-f_n\left(X+\frac{s_j}{n}\right)\right|  \right]^2  \leq  \mathbb E_X \left[\left| f_n\left(X+\frac{t_j}{n}\right)-f_n\left(X+\frac{s_j}{n}\right)\right|^2  \right]}.
 \end{array}
\]
By Lemma \ref{lem.lamperti}, we thus get for $1 \leq j \leq M$
\[
\displaystyle{ \mathbb E_X \left[\left| f_n\left(X+\frac{t_j}{n}\right)-f_n\left(X+\frac{s_j}{n}\right)\right|  \right]^2} \displaystyle{\leq   12 \times |t_j-s_j|^2 \times ||f'||_2^2 \left(  \frac{1}{n} \sum_{k=1}^n a_k^2   \right).}
\]
Gathering the previous estimates we get indeed that $\mathbb P-$almost surely, there exists a constant $C=C(\omega)$ such that 
\[
\left| \Esp_X \left[e^{iZ_n(X,t,\lambda)} \right] -\Esp_X \left[e^{iZ_n(X,s,\mu)} \right]  \right|\leq C\left(\sum_{j=1}^M \left| \lambda_j-\mu_j\right| +|\lambda_j|  \times|t_j - s_j|\right), 
\]
hence the result.

\subsection{Proof of Lemma \ref{lem.varZ}}\label{sec.varZ}
We give now the proof of Lemma \ref{lem.varZ}, which consists in showing that 
\[
\lim_{n \to +\infty}\Esp_X\left[e^{-\frac{1}{2}\Esp[Z_n(X,t,\lambda)^2]}\right]=\exp\left(-\frac{1}{2}\sum_{i,j=1}^{M} \lambda_i \lambda_j \, \rho(t_i-t_j) \right),
\]
where we recall that 
\[
Z_n(X,t,\lambda):=\sum_{i=1}^{M}\lambda_i g_n(t_i) =\sum_{i=1}^{M}\lambda_i f_n\left(X+\frac{t_i}{n}\right),
\]
and 
\[
\begin{array}{ll}
\rho(u) & \displaystyle :=\frac{1}{u}\int_0^{u}(f\ast \widetilde{f})(x)dx=\frac{1}{u}\int_0^{u}\left(\sum_{p\in \Z}\widehat{f}(p)\widehat{f}(-p) e^{i p x}\right) dx\\ 
\\
& \displaystyle= \sum_{p\in \Z}\widehat{f}(p)\widehat{f}(-p) \times \int_0^{1}e^{ip u x}dx.
\end{array}
\]
By independence of the coefficients $(a_k)$ we have 
$$\Esp[Z_n(X,t,\lambda)^2]= \sum_{i,j=1}^{M}\lambda_i \lambda_j \frac{1}{n}\sum_{k=1}^{n}f\left(kX+\frac{kt_i}{n}\right)f\left(kX+\frac{kt_j}{n}\right).$$
Since the exponential is Lipschitz, in order to prove Lemma \ref{lem.varZ} it is sufficient to establish that, for each fixed $1\leq i,j\leq M$, we have 
\begin{equation}\label{eqn.asymp}
R_{n}(i,j):= \mathbb E_X \left[\left|  \frac{1}{n}\sum_{k=1}^{n}f\left(kX+\frac{kt_i}{n}\right)f\left(kX+\frac{kt_j}{n}\right)- \rho(t_i-t_j)\right|\right]\xrightarrow{n \to +\infty} 0 .
\end{equation}
As above, we now approximate $f$ by its Fourier sum $S_Nf$, so that by triangular inequality, we get 
\[
R_n(i,j)\leq R_{n,1}(i,j)+R_{n,2}(i,j)+R_{n,3}(i,j)
\]
with 
\[
\begin{array}{l}
R_{n,1}(i,j) = \mathbb E_X \left[\left|  \frac{1}{n}\sum_{k=1}^{n}S_N f\left(kX+\frac{kt_i}{n}\right)S_N f\left(kX+\frac{kt_j}{n}\right)- \rho(t_i-t_j)\right|\right],\\
\\
R_{n,2}(i,j) = \mathbb E_X \left[\left|  \frac{1}{n}\sum_{k=1}^{n}\left( f\left(kX+\frac{kt_i}{n}\right) -S_N f\left(kX+\frac{kt_i}{n}\right)\right) f\left(kX+\frac{kt_j}{n}\right)\right|\right],\\
\\
R_{n,3}(i,j) = \mathbb E_X \left[\left|  \frac{1}{n}\sum_{k=1}^{n} S_N f \left(kX+\frac{kt_i}{n}\right)\left( f\left(kX+\frac{kt_j}{n}\right) -S_N f\left(kX+\frac{kt_j}{n}\right)\right) \right|\right].
\end{array}
\]
On the one hand, using triangular and Cauchy--Schwarz inequalities, we get uniformly in $n$
\[
R_{n,2}(p,q) \leq || f-S_N f||_2 ||f||_2, \quad R_{n,3}(p,q) \leq || f-S_N f||_2 ||S_N f||_2\leq || f-S_N f||_2 || f||_2.
\]
On the other hand, expliciting $S_Nf$ and $\rho$, by triangular inequality, we have 
\[
\begin{array}{l}
R_{n,1}(p,q) \leq \displaystyle{\sum_{|p| \leq N} |\widehat{f}(p)|^2 \left| \frac{1}{n}\sum_{k=1}^{n} e^{i \frac{k}{n} p (t_i-t_j)} - \int_0^{1}e^{ip(t_i-t_j)x}dx\right|} \\
\\
+\displaystyle{ \sum_{|p|> N} |\widehat{f}(p)|^2 \left| \int_0^{1}e^{ip(t_i-t_j)x}dx\right|} 
\displaystyle{+ \sum_{\substack{|p,q| \leq N \\ p\neq -q}} |\widehat{f}(p)\widehat{f}(q)| \mathbb E_X \left| \frac{1}{n}\sum_{k=1}^{n} e^{i k [(p+q) X + \frac{1}{n} (p t_i+q t_j)]}\right|}.
\end{array}
\]
Using standard estimates for Riemann sums, we get 
\[
\left| \frac{1}{n}\sum_{k=1}^{n} e^{i \frac{k}{n} p (t_i-t_j)} - \int_0^{1}e^{ip(t_i-t_j)x}dx\right| \leq \frac{|p|}{n},
 \]
and by Cauchy--Schwarz inequality, if $K_n$ denotes again the standard Fej\'er kernel such that $||K_n||_1=1$
\[
\mathbb E_X \left| \frac{1}{n}\sum_{k=1}^{n} e^{i k [(p+q) X + \frac{1}{n} (p t_i+q t_j)]}\right| \leq \frac{1}{\sqrt{n}} \, \mathbb E_X \left[K_n\left((p+q)X +\frac{pt_i+q t_j}{n}\right)\right]^{1/2}\leq \frac{1}{\sqrt{n}}.
\]
As a result, we get 
\[
\begin{array}{ll}
R_{n,1}(p,q) & \leq \displaystyle{\frac{1}{n} \times \sum_{|p| \leq N} |p| \, |\widehat{f}(p)|^2 + \sum_{|p|> N} |\widehat{f}(p)|^2 + \frac{1}{\sqrt{n}} \left(\sum_{p \in \mathbb Z}  |\widehat{f}(p)|\right)^2}\\
\\
& \displaystyle{\leq \frac{||f'||_2^2}{n} + || f-S_N f||_2^2 + \frac{2 ||f'||_2^2}{\sqrt{n}}.}
\end{array}
\]
Letting both $n$ and $N$ go to infinity, we deduce that for all $1\leq i,j\leq M$, the asymptotics \eqref{eqn.asymp} indeed holds, hence the result.
\subsection{Proof of Lemma \ref{lem.delta2}} \label{sec.delta2}
Let us give the proof of Lemma \ref{lem.delta2} which asserts that there exists a constant $C$ such that $\mathbb P-$almost surely
\[
\widetilde{\Delta}_n:=\mathbb E\left[ \left| \Esp_X \left[e^{iZ_n(X,t,\lambda)} \right]  - \Esp_X\left[e^{-\frac{1}{2}\Esp[Z_n(X,t,\lambda)^2]}\right]\right|^2\right] \leq \frac{C}{\sqrt{n}},
\]
where we recall that 
\[
Z_n(X,t,\lambda)= \frac{1}{\sqrt{n}} \sum_{k=1}^n a_k \, \beta_{k,n}(X), \;\; \;\text{with}\;\;\; \beta_{k,n}(X):=\sum_{j=1}^{M}\lambda_j f\left(kX+\frac{kt_j}{n}\right).
\]
The proof of similar to the one of Lemma \ref{lem.delta}. Namely developping the square under the expectation, if $Y$ an independent copy of $X$, by inverting the sums, we obtain
\begin{eqnarray*}
\widetilde{\Delta}_n&=&\Esp_{X,Y}\left[\widetilde{\Delta}_{n,1}\right]+\Esp_{X,Y}\left[\widetilde{\Delta}_{n,2}\right]+\Esp_{X,Y}\left[\widetilde{\Delta}_{n,3}\right],
\end{eqnarray*}
where, writing $Z_n(X)=Z_n(X,t,\lambda)$ to simplify the expressions, we have set
\[
\begin{array}{ll}
\widetilde{\Delta}_{n,1}:=& \displaystyle{\Esp[e^{i(Z_n(X)-Z_n(Y))}]-e^{-\frac{1}{2}\Esp[(Z_n(X)-Z_n(Y))^2]}} ,\\
\\
\widetilde{\Delta}_{n,2}:=& \displaystyle{e^{-\frac{1}{2}\Esp[(Z_n(X)-Z_n(Y))^2]} -   e^{-\frac{1}{2}(\Esp[(Z_n(X)^2]+\Esp[(Z_n(X)^2])}},\\
\\
\widetilde{\Delta}_{n,3}:=&\displaystyle{-\left(\Esp\left[e^{i Z_n(X)}\right]-e^{-\frac{1}{2}\Esp[(Z_n(X)^2]}\right)e^{-\frac{1}{2}\Esp[(Z_n(Y)^2]}} \\
& \displaystyle{- \left(\Esp\left[e^{-i Z_n(Y)}\right]-e^{-\frac{1}{2}\Esp[(Z_n(Y)^2]}\right)e^{-\frac{1}{2}\Esp[(Z_n(X)^2]}}.
\end{array}\]
Proceeding exactly as the proof of Lemma \ref{lem.delta} with a Taylor expansion at zero of the characteristic function $\phi$ of $a_1$, the terms $\widetilde{\Delta}_{n,1}(t)$ and $\widetilde{\Delta}_{n,3}(t)$ can be controlled as follows, for $n$ large enough
\[
|\widetilde{\Delta}_{n,1} |\leq \frac{8 C ||f|_{\infty}^3}{\sqrt{n}}\left( \sum_{i=1}^M |\lambda_i|\right)^3, \qquad |\widetilde{\Delta}_{n,3} |\leq \frac{ C ||f|_{\infty}^3}{\sqrt{n}}\left( \sum_{i=1}^M |\lambda_i|\right)^3.
\]
Using the fact that the exponential is Lipschitz, we have otherwise
\[
\mathbb E_{X,Y} \left[|\widetilde{\Delta}_{n,2}| \right] \leq \mathbb E_{X,Y} \left[ |\mathbb E [Z_n(X)Z_n(Y)]| \right].
\]
and the conclusion of Lemma \ref{lem.delta2} then follows from the next statement, which is multidimensional analogue of Lemma \ref{lem.decorel} above.
\begin{lemma}{Suppose that the base function $f$ satisfies the condition $\mathbf{(H1)}$}, then there exists a constant $C=C(||f||_{H^1})$ such that
\[
\mathbb E_{X,Y} \left[ \left|\mathbb E[Z_n(X)Z_n(Y)] \right|\right]\leq \frac{C}{\sqrt{n}}.
\]\label{lem.decorel2}
\end{lemma}

\begin{proof}[Proof of Lemma \ref{lem.decorel2}]
The proof follows readily the one of Lemma \ref{lem.decorel} detailed in Section \ref{sec.appdelta} above. We have 
\[
\begin{array}{ll}
\displaystyle{\mathbb E[Z_n(X)Z_n(Y)] } &  \displaystyle{= \frac{1}{n}\sum_{k=1}^n \beta_{k,n}(X)\beta_{k,n}(Y) }\\
\\ & \displaystyle{=\sum_{1\leq i,j\leq M} \lambda_i \lambda_j \left(\frac{1}{n}\sum_{k=1}^n f\left(kX+\frac{kt_i}{n}\right)f\left(kY+\frac{kt_j}{n}\right)\right)} .
\end{array}
\]
By the triangular inequality, it is thus sufficient to establish that, for any fixed $1\leq i,j\leq M$, we have 
\[
\mathbb E_{X,Y}\left[ \left|\frac{1}{n}\sum_{k=1}^n f\left(kX+\frac{kt_i}{n}\right)f\left(kY+\frac{kt_j}{n}\right)\right|\right] \leq \frac{C}{\sqrt{n}}.
\]
One then proceed exactly as in the proof of Lemma \ref{lem.decorel} by approximating $f$ by its Fourier sum $S_N f$ to deduce that 
\[
\mathbb E_{X,Y}\left[ \left|\frac{1}{n}\sum_{k=1}^n f\left(kX+\frac{kt_i}{n}\right)f\left(kY+\frac{kt_j}{n}\right)\right|\right]\leq 2\, ||f||_2 \, ||f-S_Nf ||_2 + \frac{2 ||f'||_2^2}{\sqrt{n}}.
\]
Choosing $N=\lfloor\sqrt{n}\rfloor$ then yields the result.
\end{proof}

\subsection{Proof of Lemma \ref{lma.ssuite.est}} \label{sec.ssuite.est}
This section is devoted to the proof of Lemma \ref{lma.ssuite.est}. 
Namely, writing again $Z_n(X)=Z_n(X,t,\lambda)$ to simplify the expressions, where we recall that 
\[
Z_n(X,t,\lambda)= \frac{1}{\sqrt{n}} \sum_{k=1}^n a_k \, \beta_{k,n}(X), \;\; \;\text{with}\;\;\; \beta_{k,n}(X)=\sum_{j=1}^{M}\lambda_j f\left(kX+\frac{kt_j}{n}\right),
\]
our goal is to show that as the integer $m$ goes to infinity, if $n$ is the unique integer such that  $n^{\gamma}<m\leq (n+1)^\gamma$, then $\Prob-$almost surely, we have for all $0<\alpha<1/2$
\[
\left|\Esp_X \left[e^{iZ_{n^\gamma}(X)}\right]-\Esp_X\left[e^{i Z_m(X)}\right]\right| =O\left(\left( 1-\frac{n^{\gamma}}{m}\right)^{\alpha}\right).
\]
Using again the fact that the exponential is Lipschitz, we have
\[
\left| \Esp_X\left[e^{i Z_{n^\gamma}(X)}\right]-\Esp_X\left[e^{i Z_m(X)}\right]\right| \leq  \Esp_X \left[|Z_{n^\gamma}(X)-Z_m(X)|\right].
\]
Expliciting the absolute value, we get by the triangular inequality
\[
\Esp_X \left[|Z_{n^\gamma}(X)-Z_m(X)|\right] \leq \Esp_X\left[|U|+|V|+|W|\right],
\]
where
\[
\begin{array}{ll}
\displaystyle U:=\left(1-\sqrt{\frac{n^{\gamma}}{m}}\right)\left(\frac{1}{\sqrt{n^{\gamma}}} \sum_{k=1}^{n^{\gamma}}a_k \beta_{k,n^{\gamma}}(X) \right)\\
\displaystyle V:=\sqrt{1-\frac{n^{\gamma}}{m}}  \left(\frac{1}{\sqrt{m-n^{\gamma}}}\sum_{k=n^{\gamma}+1}^{m}a_k\beta_{k,m}(X)\right),\\
\displaystyle W:=\frac{1}{\sqrt{m}}\left(\sum_{k=1}^{n^{\gamma}}a_k\left(\beta_{k,n^{\gamma}}(X)-\beta_{k,m}(X)\right)\right).\\
\end{array}
\]
Applying Cauchy--Schwarz inequality, we have 
\[
\Esp_X [|U|]^2 \leq 
\left(1-\sqrt{\frac{n^{\gamma}}{m}}\right)^2\left( \frac{1}{n^{\gamma}} \sum_{k,\ell=1}^{n^{\gamma}} a_k a_{\ell} \, \mathbb E_X [\beta_{k,n^{\gamma}}(X)\beta_{\ell,n^{\gamma}}(X)] \right)
\]
with 
\[
\mathbb E_X [\beta_{k,n^{\gamma}}(X)\beta_{\ell,n^{\gamma}}(X)]=\sum_{i,j=1}^M \lambda_i \lambda_j \mathbb E_X \left[ f\left(kX+\frac{kt_i}{n^{\gamma}}\right)f\left(\ell X+\frac{\ell t_j}{n^{\gamma}}\right) \right]. 
\]
Proceeding as in the proof of Lemma  \ref{lem.controlL2}, i.e. approximating $f$ by its Fourier sum $S_N f$, one then gets, uniformly in $k,\ell$ and $i,j$
\[
\begin{array}{ll}
\displaystyle{\left| \mathbb E_X \left[ f\left(kX+\frac{kt_i}{n^{\gamma}}\right)f\left(\ell X+\frac{\ell t_j}{n^{\gamma}}\right) \right]\right| } & \displaystyle{\leq    \left| \mathbb E_X \left[ S_N f\left(kX+\frac{kt_i}{n^{\gamma}}\right)S_Nf\left(\ell X+\frac{\ell t_j}{n^{\gamma}}\right) \right] \right|}\\
\\
&\displaystyle{+2 ||f-S_N  f||_2 \times ||f||_2}.
\end{array}
\]
Besides, by Minkowski inequality, we have 
\[
\begin{array}{rl}
S: = &\displaystyle{\sum_{i,j=1}^M \lambda_i \lambda_j \frac{1}{n^{\gamma}} \sum_{k,l=1}^{n^{\gamma}} a_k a_{\ell} \,  \mathbb E_X \left[ S_N f\left(kX+\frac{kt_i}{n^{\gamma}}\right)S_Nf\left(\ell X+\frac{\ell t_j}{n^{\gamma}}\right) \right]} \\
\\
 = &\displaystyle{  \mathbb E_X \left[\left|\sum_{|p| \leq N} \widehat{f}(p) \sum_{j=1}^M \lambda_j \frac{1}{\sqrt{n^{\gamma}}}  \sum_{k=1}^{n^{\gamma}} a_k e^{i p \left( kX+\frac{kp t_j}{n^{\gamma}}\right)} \right|^2 \right] }\\
 \\
  \leq  &\displaystyle{  \left[ \sum_{|p| \leq N} |\widehat{f}(p)| \sum_{j=1}^M |\lambda_j|  \mathbb E_X \left[ \left| \frac{1}{\sqrt{n^{\gamma}}}  \sum_{k=1}^{n^{\gamma}} a_k e^{i p \left( kX+\frac{kp t_j}{n^{\gamma}}\right)} \right|^2\right]^{1/2} \right]^2 }\\
  \\
   =  &\displaystyle{  \left[ \sum_{|p| \leq N} |\widehat{f}(p)| \sum_{j=1}^M |\lambda_j|   \left( \frac{1}{n^{\gamma}}  \sum_{k=1}^{n^{\gamma}} a_k^2\right)^{1/2} \right]^2 } \leq 2 \left( \frac{1}{n^{\gamma}}  \sum_{k=1}^{n^{\gamma}} a_k^2\right) ||f'||_2^2 ||\lambda||_1^2.
\end{array}
\]
As a result, we get that $\Esp_X [|U|]^2$ is upper bounded by 
\[
2\left(1-\sqrt{\frac{n^{\gamma}}{m}}\right)^2 ||\lambda||_1^2
\left[  \left( \frac{1}{n^{\gamma}} \sum_{k,l=1}^{n^{\gamma}} |a_k a_{\ell}|\right)  ||f-S_N  f||_2 \times ||f||_2 +  || f'||_2^2\left( \frac{1}{n^{\gamma}}  \sum_{k=1}^{n^{\gamma}} a_k^2\right)\right].
\]
Choosing $N=n^{\gamma}$, so that $||f-S_N  f||_2 =O(1/n^{\gamma})$, by the law of large numbers, we conclude that there exists a constant $C=C(\omega)$ such that $\mathbb P-$almost surely 
\[
\Esp_X [|U|]^2 \leq C \left(1-\sqrt{\frac{n^{\gamma}}{m}}\right)^2.
\]
By the exact same arguments,  there exists a constant $C=C(\omega)$ such that $\mathbb P-$almost surely 
\[
\Esp_X [|V|]^2 \leq C \left(1-\frac{n^{\gamma}}{m} \right).
\]
For the last term, using again Cauchy--Schwarz inequality, we have 
\[
\Esp_X [|W|]^2  \displaystyle{\leq \frac{1}{m} \sum_{k,\ell=1}^{n^{\gamma}} a_k a_{\ell} \, \mathbb E_X [(\beta_{k,n^{\gamma}}(X)-\beta_{k,m}(X) )(\beta_{\ell,n^{\gamma}}(X)-\beta_{\ell,m}(X))]}.
\]
Moreover, a direct computation yields, uniformly in $1 \leq j \leq M$ and $1 \leq k \leq n^{\gamma}$ 
\[
\begin{array}{ll}
\displaystyle{\mathbb E_X \left[\left|S_N f\left(kX+\frac{kt_j}{n^{\gamma}}\right)-S_N f\left(kX+\frac{kt_j}{m}\right) \right|\right] }\leq \displaystyle{\left( \sum_{|p|\leq N} |\widehat{f}(p)|^2 \left| e^{ip \frac{k }{n^{\gamma}}t_j} - e^{ip \frac{k }{m}t_j} \right|^2\right)^{1/2}}\\
\\
\leq \displaystyle{\left( \sum_{|p|\leq N} |p \widehat{f}(p)|^2 \left|  \frac{k}{n^{\gamma}} \left( 1-\frac{n^{\gamma}}{m}\right)\right|^2\right)^{1/2}\leq ||f'||_2   \left( 1-\frac{n^{\gamma}}{m}\right)}.
\end{array}
\]
As a result, approximating $f$ by its Fourier sum $S_N f$, we get this time 
\[
\begin{array}{ll}
\Esp_X [|W|]^2 \displaystyle{\leq  4 ||\lambda||_1^2 \left( ||f-S_N f||_2^2 +   ||f-S_N f||_2 ||f'||_2   \left( 1-\frac{n^{\gamma}}{m}\right)\right)\left( \frac{1}{m} \sum_{k,l=1}^{n^{\gamma}} |a_k a_{\ell}|\right)+ T}
\end{array}
\]
where 
\[
\begin{array}{ll}
T:= \displaystyle{  \mathbb E_X \left[\left|\sum_{|p| \leq N} \widehat{f}(p) \sum_{i=1}^M \lambda_i \frac{1}{\sqrt{m}}  \sum_{k=1}^{n^{\gamma}} a_k e^{i p kX} \left( e^{i p \frac{k }{n^{\gamma}}t_j } - e^{i  p \frac{k}{m}t_j}\right) \right|^2 \right] }
\\
\\
  \leq  \displaystyle{  \left[ \sum_{|p| \leq N} |\widehat{f}(p)| \sum_{i=1}^M |\lambda_i|  \mathbb E_X \left[ \left| \frac{1}{\sqrt{m}}  \sum_{k=1}^{n^{\gamma}} a_k e^{i p kX} \left( e^{i p t_j \frac{k }{n^{\gamma}}} - e^{i  p t_j\frac{k}{m}}\right)\right|^2\right]^{1/2} \right]^2 }
  \\
\\
= \displaystyle{  \left[ \sum_{|p| \leq N} |\widehat{f}(p)| \sum_{i=1}^M |\lambda_i|  \left( \frac{1}{m}  \sum_{k=1}^{n^{\gamma}} a_k^2  \left| e^{ip \frac{k }{n^{\gamma}}t_j} - e^{ip \frac{k }{m}t_j} \right|^2\right)^{1/2} \right]^2 }.
\end{array}
\]
Since
\[
\left| e^{ip \frac{k }{n^{\gamma}}t_j} - e^{ip \frac{k }{m}t_j} \right|^2 \leq   |p|^{2}||t||_{\infty}^{2} \left( 1-\frac{n^{\gamma}}{m}\right)^{2},
\]
we obtain
\[
T \leq  \left( \sum_{|p| \leq N} |\widehat{f}(p)| \times |p| \right)^2 ||\lambda||_1^2 \times  ||t||_{\infty}^{2} \left( 1-\frac{n^{\gamma}}{m}\right)^{2}  \left( \frac{1}{m}  \sum_{k=1}^{n^{\gamma}} a_k^2 \right).
\]

Choosing $N=n^{\gamma}$, so that $||f-S_N  f||_2 =O(1/n^{\gamma})$, by the law of large numbers, we conclude as above that  there exists a constant $C=C(\omega)$ such that $\mathbb P-$almost surely 
\[
\Esp_X [|W|]^2 \leq C \left( 1-\frac{n^{\gamma}}{m}\right)^{2}, 
\]
hence the result.
\subsection{Proof of Lemma \ref{lem.lamperti}}\label{sec.lem.lamperti}
We now give the proof of Lemma \ref{lem.lamperti} ensuring in particular the tightness of the family of processes $(g_n)_{n \geq 1}$ in the $\mathcal C^1-$topology. For $t \neq s$ we have
\[
\begin{array}{ll}
\displaystyle{ \mathbb E_X \left[\left| g_n\left(t \right)-g_n\left(s\right)\right|^2  \right]  =  \mathbb E_X \left[\left| f_n\left(X+\frac{t}{n}\right)-f_n\left(X+\frac{s}{n}\right)\right|^2  \right]} \\
 \\
 \displaystyle{= \mathbb E_X \left[\left| \frac{1}{\sqrt{n}} \sum_{k=1}^n a_k  \left( f\left(kX+\frac{kt}{n}\right)-f\left(kX+\frac{ks}{n}\right)\right)\right|^2  \right]}.
 \end{array}
\]
As in the proofs of Lemmas \ref{lem.delta} and \ref{lem.controlL2} above, we can then approximate $f$ by its Fourier sum $S_Nf$. Doing so, we obtain after applying Cauchy--Schwarz inequality
\[
\begin{array}{ll}
\displaystyle{ \mathbb E_X \left[\left| g_n\left(t \right)-g_n\left(s\right)\right|^2  \right]   \leq   4 ||f-S_Nf||_2^2\left( \frac{1}{n} \sum_{k,\ell=1}^n |a_k a_{\ell}|\right)}\\
\\
\displaystyle{ + \, 4 ||f-S_Nf||_2 \left( \frac{1}{n} \sum_{k,\ell=1}^n |a_k a_{\ell}|\, \mathbb E_X\left[\left| S_N f\left(kX+\frac{kt}{n}\right)-S_N f\left(kX+\frac{ks}{n}\right) \right|^2  \right]^{1/2}\right)}\\
 \\
 \displaystyle{ +\mathbb E_X \left[\left| \frac{1}{\sqrt{n}} \sum_{k=1}^n a_k  \left( S_N f\left(kX+\frac{kt}{n}\right)-S_N f\left(kX+\frac{ks}{n}\right)\right)\right|^2  \right]}.
 \end{array}
\]
Using Parseval inequality, a direct computation gives uniformly in $1\leq k\leq n$
\[
\begin{array}{ll}
 \displaystyle{\mathbb E_X\left[\left| S_N f\left(kX+\frac{kt}{n}\right)-S_N f\left(kX+\frac{ks}{n}\right) \right|^2  \right]} & =  \displaystyle{\sum_{|p|\leq N} |\widehat{f}(p)|^2 \left| e^{ip \frac{k}{n} t} - e^{ip \frac{k}{n} s}\right|^2 }\\
 \\
 &  \displaystyle{\leq ||f'||_2^2 \times |t-s|^2.}
 \end{array}
\]
Besides, by Minkowski inequality, we have 
\[
\begin{array}{ll}
A_n & := \displaystyle{ \mathbb E_X \left[\left| \frac{1}{\sqrt{n}} \sum_{k=1}^n a_k  \left( S_N f\left(kX+\frac{kt}{n}\right)-S_N f\left(kX+\frac{ks}{n}\right)\right)\right|^2  \right]} \\
\\
& = \displaystyle{ \mathbb E_X \left[\left| \sum_{|p|\leq N} \widehat{f}(p) \frac{1}{\sqrt{n}} \sum_{k=1}^n a_k  e^{ip kX} 
\left( e^{i p\frac{k}{n}t} -e^{i p\frac{k}{n}s} \right)\right|^2  \right]} \\
\\
& \leq  \displaystyle{ \left(  \sum_{|p|\leq N} |\widehat{f}(p)|  \, \mathbb E_X \left[ \left| \frac{1}{\sqrt{n}} \sum_{k=1}^n a_k  e^{ip kX} 
\left( e^{i p\frac{k}{n}t} -e^{i p\frac{k}{n}s} \right)\right|^2  \right]^{1/2} \right)^2}\\
\\
& =  \displaystyle{ \left(  \sum_{|p|\leq N} |\widehat{f}(p)|  \,  \left[\frac{1}{n} \sum_{k=1}^n a_k^2  
\left| e^{i p\frac{k}{n}t} -e^{i p\frac{k}{n}s} \right|^2  \right]^{1/2} \right)^2}.
\\
\\
& \leq 4 \times \displaystyle{\left(  \sum_{|p|\leq N} |\widehat{f}(p)| \, |p| \right)^2 \,  |t -s|^{2} \left(\frac{1}{n} \sum_{k=1}^n a_k^2   \right)} \\
\\
&  \displaystyle{ \leq4  ||f'||_2^2\times   |t -s|^{2} \times \left(\frac{1}{n} \sum_{k=1}^n a_k^2  \right) }.
  \end{array} 
\]
Gathering all the previous estimates and choosing $N$ large enough such that 
\[
||f-S_Nf||_2 \leq \frac{1}{n} \times |t-s| \times ||f'||_2,
\]
we thus get 
\[
\begin{array}{ll}
\displaystyle{ \mathbb E_X \left[\left| g_n\left(t \right)-g_n\left(s\right)\right|^2  \right] }& \displaystyle{  \leq   4\times  |t-s|^2 \times ||f'||_2 \left(  \frac{1}{n} \sum_{k=1}^n a_k^2 +\left( 1 +\frac{1}{n}\right) \left( \frac{1}{n} \sum_{k=1}^n |a_k|\right)^2  \right)}\\
\\
&  \displaystyle{\leq   12 \times |t-s|^2 \times ||f'||_2^2 \left(  \frac{1}{n} \sum_{k=1}^n a_k^2   \right).}
 \end{array}
\]
In the same manner,  if the base function $f$ is of class $\mathcal C^2$, we have 
\[
\begin{array}{ll}
\displaystyle{ \mathbb E_X \left[\left| g_n'\left(t \right)-g_n'\left(s\right)\right|^2  \right] } & 
 \displaystyle{= \mathbb E_X \left[\left| \frac{1}{\sqrt{n}} \sum_{k=1}^n a_k \times \frac{k}{n}  \left( f'\left(kX+\frac{kt}{n}\right)-f'\left(kX+\frac{ks}{n}\right)\right)\right|^2  \right]}.
 \end{array}
\]
Replacing $a_k$ by $a_k \times k/n$ et $f$ by $f'$ in the proof above, we get similarly 
\[
\begin{array}{ll}
\displaystyle{ \mathbb E_X \left[\left| g_n'\left(t \right)-g_n'\left(s\right)\right|^2  \right] } & 
 \displaystyle{\leq   12 \times |t-s|^2 \times ||f''||_2^2 \left(  \frac{1}{n} \sum_{k=1}^n a_k^2   \right).}
 \end{array}
\]

\subsection{Proof of Lemma \ref{lm.est.bir.gen}}
The proof follows the same lines as the one of Lemma \ref{lem.controlL2}.
Developing the square, we have 
\[
\mathbb E_X [g_n^{(p)}(t)^2]=  \frac{1}{n} \sum_{k, \ell=1}^n \frac{k^{p}\ell^{p}}{n^{2p}}a_k a_{\ell} \mathbb E_X\left[ f^{(p)}\left(kX+\frac{kt}{n}\right)f^{(p)}\left(\ell X+\frac{kt}{n}\right)\right]. 
\]
As in the proof of Lemma \ref{lem.decorel}, we can decompose
\[
\begin{array}{ll}
I_{p,k,\ell}(X)& :=\displaystyle f^{(p)}\left(kX+\frac{kt}{n}\right)f^{(p)}\left(\ell X+\frac{kt}{n}\right) \\
\\
& = \displaystyle \left[f^{(p)}\left(kX+\frac{kt}{n}\right)-S_N f^{(p)}\left(kX+\frac{kt}{n}\right) \right]f^{(p)}\left(\ell X+\frac{\ell t}{n}\right) \\
\\
& \displaystyle +\,  S_N f^{(p)}\left(kX+\frac{kt}{n}\right) \left[ f^{(p)}\left(\ell X+\frac{\ell t}{n}\right)-S_N f^{(p)}\left(\ell X+\frac{\ell t}{n}\right)  \right]\\
\\
&\displaystyle + \, S_N f^{(p)}\left(kX+\frac{kt}{n}\right) S_N f^{(p)}\left(\ell X+\frac{\ell t}{n}\right),
\end{array}
\]
so that 
\[
\begin{array}{ll}
\displaystyle{\mathbb E_X [g_n^{(p)}(t)^2]}  = \frac{1}{n} \sum_{k, \ell=1}^n \frac{k^{p}\ell^{p}}{n^{2p}}a_k a_{\ell}  \mathbb E_X \left[ I_{p,k,\ell}(X)\right] \\
\\
   \displaystyle = \frac{1}{n} \sum_{k, \ell=1}^n \frac{k^{p}\ell^{p}}{n^{2p}}a_k a_{\ell} \mathbb E_X \left[ \left[f^{(p)}\left(kX+\frac{kt}{n}\right)-S_N f^{(p)}\left(kX+\frac{kt}{n}\right) \right]f^{(p)}\left(\ell X+\frac{\ell t}{n}\right) \right] \\
\\
 \displaystyle+\frac{1}{n} \sum_{k, \ell=1}^n \frac{k^{p}\ell^{p}}{n^{2p}}a_k a_{\ell} \mathbb E_X \left[ S_N f^{(p)}\left(kX+\frac{kt}{n}\right) \left[ f^{(p)}\left(\ell X+\frac{\ell t}{n}\right)-S_N f^{(p)}\left(\ell X+\frac{\ell t}{n}\right)  \right]\right]\\
 \\
 +\displaystyle \frac{1}{n} \sum_{k, \ell=1}^n \frac{k^p\ell^p}{n^{2p}}a_k a_{\ell} \mathbb E_X \left[ S_N f^{(p)}\left(kX+\frac{kt}{n}\right) S_N f^{(p)}\left(\ell X+\frac{\ell t}{n}\right)\right].
\end{array}
\]
By the triangular inequality and Cauchy--Schwarz inequality, we have
\begin{eqnarray*}
&&\left|  \frac{1}{n} \sum_{k, \ell=1}^n \frac{k^{p}\ell^{p}}{n^{2p}}a_k a_{\ell} \mathbb E_X \left[ \left[f^{(p)}\left(kX+\frac{kt}{n}\right)-S_N f^{(p)}\left(kX+\frac{kt}{n}\right) \right]f^{(p)}\left(\ell X+\frac{\ell t}{n}\right) \right] \right|\\
& \leq &\left(  \frac{1}{n} \sum_{k, \ell=1}^n |a_k a_{\ell} |\right) ||f^{(p)}-S_N  f^{(p)}||_2 \times ||f^{(p)}||_2, 
\end{eqnarray*}
and similarly 
\begin{eqnarray*}
&&\left| \frac{1}{n} \sum_{k, \ell=1}^n \frac{k^{p}\ell^{p}}{n^{2p}}a_k a_{\ell} \mathbb E_X \left[ \left[f^{(p)}\left(kX+\frac{kt}{n}\right)-S_N f^{(p)}\left(kX+\frac{kt}{n}\right) \right]f^{(p)}\left(\ell X+\frac{\ell t}{n}\right) \right]  \right|\\
& \leq &\left(  \frac{1}{n} \sum_{k, \ell=1}^n |a_k a_{\ell} |\right) ||f^{(p)}-S_N  f^{(p)}||_2 \times ||f^{(p)}||_2.
\end{eqnarray*}
The last term can be rewritten as a square, and Minkowski inequality then yields
\[
\begin{array}{l}
\displaystyle \frac{1}{n} \sum_{k, \ell=1}^n \frac{k^p\ell^p}{n^{2p}}a_k a_{\ell} \mathbb E_X \left[ S_N f^{(p)}\left(kX+\frac{kt}{n}\right) S_N f^{(p)}\left(\ell X+\frac{\ell t}{n}\right)\right] \\
\\
= \displaystyle{\mathbb E_X \left[ \left|\frac{1}{\sqrt{n}} \sum_{k=1}^n \frac{k^p}{n^{p}}a_k S_N f^{(p)} \left(kX+\frac{kt}{n}\right)\right|^2\right]} \\
\\
=  \displaystyle{ \mathbb E_X \left[ \left|\sum_{|r|\leq N} \widehat{f^{(p)}}(r) \left( \frac{1}{\sqrt{n}} \sum_{k=1}^n \frac{k^p}{n^p}a_k e^{i kr (X+t/n)}\right)\right|^2\right]} \\
\\
\leq \displaystyle{\left[ \sum_{|r|\leq N} |\widehat{f^{(p)}}(r)|  \times \mathbb E_X \left[ \left|  \frac{1}{\sqrt{n}} \sum_{k=1}^n \frac{k^p}{n^p}a_k e^{i kr (X+t/n)}\right|^2\right]^{1/2} \right]^2} \\
\\
 = \displaystyle{\left( \frac{1}{n} \sum_{k=1}^n \frac{k^{2p}}{n^{2p}}a_k^2\right) \left( \sum_{|r|\leq N} |\widehat{f^{(p)}}(r)| \right)^2 \leq \left( \frac{1}{n} \sum_{k=1}^n a_k^2\right) \times 2 || f^{(p+1)}||_2^2.}
\end{array}
\]
Therefore, uniformly on $t\in [0,2\pi]$,
\[
\mathbb E_X [g_n^{(p)}(t)^2] \leq \left(  \frac{2}{n} \sum_{k, \ell=1}^n |a_k a_{\ell} |\right) ||f^{(p)}-S_N  f^{(p)}||_2 \times ||f^{(p)}||_2 + \left( \frac{1}{n} \sum_{k=1}^n a_k^2\right) \times 2 || f^{(p+1)}||_2^2.
\]
Choosing $N=n$, so that $||f^{(p)}-S_N  f^{(p)}||_2  =O(1/n)$, by the law of large numbers, we conclude that there exists a constant $C=C(f,\omega)$ such that $\mathbb P-$almost surely 
\[
\sup_{n \geq 1} \sup_{t\in [0,2\pi]}\mathbb E_X [g_n^{(p)}(t)^2] \leq C.
\]
Taking the expectation with respect to $\mathbb P$, we have moreover 
\[
\sup_{n \geq 1} \sup_{t\in [0,2\pi]}\mathbb E \mathbb E_X [g_n^{(p)}(t)^2] \leq  ||f^{(p)}||_2 + 2 || f^{(p+1)}||_2^2.
\]


\bibliographystyle{plain}

\bibliography{biblio3}

@book{Ada03,
  title={Sobolev Spaces (Pure and applied mathematics; v. 140)},
  author={Adams, Robert A and Fournier, John JF},
  year={2003},
  publisher={Elsevier}
}

@Article{ADP19,
 Author = {J\"urgen {Angst} and Federico {Dalmao} and Guillaume {Poly}},
 Title = {{On the real zeros of random trigonometric polynomials with dependent coefficients}},
 FJournal = {{Proceedings of the American Mathematical Society}},
 Journal = {{Proc. Am. Math. Soc.}},
 ISSN = {0002-9939; 1088-6826/e},
 Volume = {147},
 Number = {1},
 Pages = {205--214},
 Year = {2019},
 Publisher = {American Mathematical Society (AMS), Providence, RI},
 Language = {English},
 MSC2010 = {26C10 42A05 12D10 30C15},
 Zbl = {1406.26007}
}

@Article{AP15,
      title={Universality of the mean number of real zeros of random trigonometric polynomials under a weak {C}ram\'er condition}, 
      author={Jürgen Angst and Guillaume Poly},
      year={2015},
      eprint={1511.08750},
      archivePrefix={arXiv},
      Volume = {arXiv:1511.08750},
      primaryClass={math.PR}
}

@article{AP20IMRN,
    author = {Angst, Jürgen and Poly, Guillaume},
    title = "{On the Zeros of Non-Analytic Random Periodic Signals}",
    journal = {International Mathematics Research Notices},
    year = {2020},
    month = {08},
    issn = {1073-7928},
    doi = {10.1093/imrn/rnaa201},
    url = {https://doi.org/10.1093/imrn/rnaa201},
    note = {rnaa201},
    eprint = {https://academic.oup.com/imrn/advance-article-pdf/doi/10.1093/imrn/rnaa201/33684058/rnaa201.pdf},
}

@article{APP22,
 author = {Angst, J{\"u}rgen and Pautrel, Thibault and Poly, Guillaume},
 title = {Real zeros of random trigonometric polynomials with dependent coefficients},
 fjournal = {Transactions of the American Mathematical Society},
 journal = {Trans. Am. Math. Soc.},
 issn = {0002-9947},
 volume = {375},
 number = {10},
 pages = {7209--7260},
 year = {2022},
 language = {English},
 doi = {10.1090/tran/8742},
 zbMATH = {7599882},
 Zbl = {1511.26017}
}

@Book{AW09,
 Author = {Jean-Marc {Aza\"{\i}s} and Mario {Wschebor}},
 Title = {{Level sets and extrema of random processes and fields}},
 ISBN = {978-0-470-40933-6/hbk; 978-0-470-43464-2/ebook},
 Pages = {xi + 393},
 Year = {2009},
 Publisher = {Hoboken, NJ: John Wiley \& Sons},
 Language = {English},
 MSC2010 = {60-02 60G60 60D05 62M40},
 Zbl = {1168.60002}
}

@Article{BCP18,
 Author = {Vlad {Bally} and Lucia {Caramellino} and Guillaume {Poly}},
 Title = {{Non universality for the variance of the number of real roots of random trigonometric polynomials}},
 FJournal = {{Probability Theory and Related Fields}},
 Journal = {{Probab. Theory Relat. Fields}},
 ISSN = {0178-8051; 1432-2064/e},
 Volume = {174},
 Number = {3-4},
 Pages = {887--927},
 Year = {2019},
 Publisher = {Springer, Berlin/Heidelberg},
 Language = {English},
 MSC2010 = {60G50 60F05}
}

@Article{DNV18,
 Author = {Yen {Do} and Oanh {Nguyen} and Van {Vu}},
 Title = {{Roots of random polynomials with coefficients of polynomial growth}},
 FJournal = {{The Annals of Probability}},
 Journal = {{Ann. Probab.}},
 ISSN = {0091-1798; 2168-894X/e},
 Volume = {46},
 Number = {5},
 Pages = {2407--2494},
 Year = {2018},
 Publisher = {Institute of Mathematical Statistics (IMS), Beachwood, OH/Bethesda, MD},
 Language = {English},
 MSC2010 = {60G99},
 Zbl = {1428.60072}
}

@article{DNN20,
 author = {Do, Yen and Nguyen, Hoi H. and Nguyen, Oanh},
 title = {Random trigonometric polynomials: universality and non-universality of the variance for the number of real roots},
 fjournal = {Annales de l'Institut Henri Poincar{\'e}. Probabilit{\'e}s et Statistiques},
 journal = {Ann. Inst. Henri Poincar{\'e}, Probab. Stat.},
 issn = {0246-0203},
 volume = {58},
 number = {3},
 pages = {1460--1504},
 year = {2022},
 language = {English},
 doi = {10.1214/21-AIHP1206},
 zbMATH = {7561791},
 Zbl = {1493.60049}
}

@Article{Dun66,
 Author = {J. E. A. {Dunnage}},
 Title = {{The number of real zeros of a random trigonometric polynomial}},
 FJournal = {{Proceedings of the London Mathematical Society. Third Series}},
 Journal = {{Proc. Lond. Math. Soc. (3)}},
 ISSN = {0024-6115; 1460-244X/e},
 Volume = {16},
 Pages = {53--84},
 Year = {1966},
 Publisher = {John Wiley \& Sons, Chichester; London Mathematical Society, London},
 Language = {English},
 Zbl = {0141.15003}
}

@Article{Fla17,
 Author = {Hendrik {Flasche}},
 Title = {{Expected number of real roots of random trigonometric polynomials}},
 FJournal = {{Stochastic Processes and their Applications}},
 Journal = {{Stochastic Processes Appl.}},
 ISSN = {0304-4149},
 Volume = {127},
 Number = {12},
 Pages = {3928--3942},
 Year = {2017},
 Publisher = {Elsevier (North-Holland), Amsterdam},
 Language = {English},
 MSC2010 = {60G99 60F05 42A05},
 Zbl = {1377.60063}
}

@Article{IKM16,
 Author = {Alexander {Iksanov} and Zakhar {Kabluchko} and Alexander {Marynych}},
 Title = {{Local universality for real roots of random trigonometric polynomials}},
 FJournal = {{Electronic Journal of Probability}},
 Journal = {{Electron. J. Probab.}},
 ISSN = {1083-6489/e},
 Volume = {21},
 Pages = {19},
 Note = {Id/No 63},
 Year = {2016},
 Publisher = {University of Washington, Department of Mathematics, Seattle, WA; Duke University, Department of Mathematics, Durham, NC},
 Language = {English},
 MSC2010 = {30C15 60F17},
 Zbl = {1361.30009}
}

@article{Mat12,
  title={The real zeros of a random algebraic polynomial with dependent coefficients},
  author={Matayoshi, Jeffrey},
  journal={The Rocky Mountain Journal of Mathematics},
  pages={1015--1034},
  year={2012},
  publisher={JSTOR}
}

@article{Muk19,
  title={Average number of real zeros of random algebraic polynomials defined by the increments of fractional Brownian motion},
  author={Mukeru, Safari},
  journal={Journal of Theoretical Probability},
  volume={32},
  number={3},
  pages={1502--1524},
  year={2019},
  publisher={Springer}
}

@article{NV18,
 author = {Nguyen, Oanh and Vu, Van},
 title = {Roots of random functions: a framework for local universality},
 fjournal = {American Journal of Mathematics},
 journal = {Am. J. Math.},
 issn = {0002-9327},
 volume = {144},
 number = {1},
 pages = {1--74},
 year = {2022},
 language = {English},
 doi = {10.1353/ajm.2022.0000},
 url = {muse.jhu.edu/article/844458},
 zbMATH = {7465466},
 Zbl = {1496.60012}
}

@Article{Pau20,
 Author = {Pautrel, Thibault},
 Title = {{New asymptotics for the mean number of zeros of random trigonometric polynomials with strongly dependent Gaussian coefficients}},
 FJournal = {{Electronic Communications in Probability}},
 Journal = {{Electron. Commun. Probab.}},
 ISSN = {1083-589X/e},
 Volume = {25},
 Pages = {13},
 Note = {Id/No 36},
 Year = {2020},
 Publisher = {University of Washington, Washington},
 Language = {English},
 MSC2010 = {60H99 60G99 26C10},
 Zbl = {1445.60052}
}

@Article{Pir19a,
 Author = {Ali {Pirhadi}},
 Title = {{Real zeros of random trigonometric polynomials with pairwise equal blocks of coefficients}},
 FJournal = {{Rocky Mountain Journal of Mathematics}},
 Journal = {{Rocky Mt. J. Math.}},
 ISSN = {0035-7596},
 Volume = {50},
 Number = {4},
 Pages = {1451--1471},
 Year = {2020},
 Publisher = {Rocky Mountain Mathematics Consortium c/o Arizona State University, Tempe, AZ},
 Language = {English},
 MSC2010 = {30C15 12D10 26C10 26C10}
}

@article{Pir19b,
      title={Real zeros of random cosine polynomials with palindromic blocks of coefficients}, 
      author={Ali Pirhadi},
      year={2021},
      journal={Analysis Mathematica},
      Volume={47},
      number={1},
      pages={175--210}
}

@article{RS01,
  title={On weak convergence of functionals on smooth random functions},
  author={Rusakov, Alexander and Seleznjev, Oleg},
  journal={Mathematical Communications},
  volume={6},
  number={2},
  pages={123--134},
  year={2001},
  publisher={Odjel za matematiku, Sveu{\v{c}}ili{\v{s}}te JJ Strossmayera u Osijeku}
}

@Article{SZ54,
 Author = {R. {Salem} and Antoni {Zygmund}},
 Title = {{Some properties of trigonometric series whose terms have random signs}},
 FJournal = {{Acta Mathematica}},
 Journal = {{Acta Math.}},
 ISSN = {0001-5962; 1871-2509/e},
 Volume = {91},
 Pages = {245--301},
 Year = {1954},
 Publisher = {International Press of Boston, Somerville, MA; Institut Mittag-Leffler, Stockholm},
 Language = {English},
 Zbl = {0056.29001}
}

@article{TV15,
  title={Local universality of zeroes of random polynomials},
  author={Tao, Terence and Vu, Van},
  journal={International Mathematics Research Notices},
  volume={2015},
  number={13},
  pages={5053--5139},
  year={2015},
  publisher={Oxford University Press}
}

@article{AMP24,
author = {Angst, Jürgen and Malicet, Dominique and Poly, Guillaume},
title = {Almost sure behavior of the critical points of random polynomials},
journal = {Bulletin of the London Mathematical Society},
volume = {56},
number = {2},
pages = {767-782},
doi = {https://doi.org/10.1112/blms.12963},
url = {https://londmathsoc.onlinelibrary.wiley.com/doi/abs/10.1112/blms.12963},
eprint = {https://londmathsoc.onlinelibrary.wiley.com/doi/pdf/10.1112/blms.12963},
year = {2024}
}

@article{AP23,
author = {J{\"u}rgen Angst and Guillaume Poly},
title = {{Fluctuations in Salem–Zygmund almost sure {C}entral {L}imit {T}heorem}},
volume = {28},
journal = {Electronic Journal of Probability},
number = {none},
publisher = {Institute of Mathematical Statistics and Bernoulli Society},
pages = {1 -- 40},
year = {2023},
doi = {10.1214/23-EJP931},
URL = {https://doi.org/10.1214/23-EJP931}
}

@article{AP21,
 author = {Angst, J{\"u}rgen and Poly, Guillaume},
 title = {Variations on {Salem}--{Zygmund} results for random trigonometric polynomials: application to almost sure nodal asymptotics},
 fjournal = {Electronic Journal of Probability},
 journal = {Electron. J. Probab.},
 issn = {1083-6489},
 volume = {26},
 pages = {36},
 note = {Id/No 156},
 year = {2021},
 language = {English},
 doi = {10.1214/21-EJP716},
 zbMATH = {7478671},
 Zbl = {1492.60074}
}

@book {MR2289012,
    AUTHOR = {Tao, Terence and Vu, Van},
     TITLE = {Additive combinatorics},
    SERIES = {Cambridge Studies in Advanced Mathematics},
    VOLUME = {105},
 PUBLISHER = {Cambridge University Press, Cambridge},
      YEAR = {2006},
     PAGES = {xviii+512},
      ISBN = {978-0-521-85386-6; 0-521-85386-9},
   MRCLASS = {11-02 (05-02 05D10 11B13 11P70 11P82 28D05 37A45)},
  MRNUMBER = {2289012},
MRREVIEWER = {Serge\u i\ V.\ Konyagin and Ilya\ D.\ Shkredov},
       DOI = {10.1017/CBO9780511755149},
       URL = {https://doi.org/10.1017/CBO9780511755149},
}

@misc{ANP24,
 author = {Angst, J{\"u}rgen and Nguyen, Oanh and Poly, Guillaume},
 title = {Roots of random trigonometric polynomials with general dependent coefficients},
 year = {2024},
 howpublished = {Preprint, {arXiv}:2409.15057 [math.{PR}] (2024)},
 url = {https://arxiv.org/abs/2409.15057},
 arXiv = {arXiv:2409.15057}
}

\end{document}